\documentclass[11pt]{amsart}

\usepackage{etex} 

\usepackage[foot]{amsaddr}

\usepackage[linktocpage]{hyperref}
\usepackage[a4paper, centering, vcentering]{geometry}
\usepackage{fullpage}
\usepackage{cancel}
\usepackage[T1]{fontenc} 
\usepackage{lmodern} 

\usepackage{mathtools} 
\mathtoolsset{showonlyrefs,showmanualtags} 

\usepackage{bbm}

\usepackage[english]{babel}
\usepackage{amsthm,amssymb,amsbsy,amsmath,amsfonts,mathrsfs,amscd}
\usepackage{cite}
\usepackage{url}

\usepackage{tikz}\usetikzlibrary{arrows}

\usepackage{color}
\usepackage[usenames,dvipsnames]{pstricks}
\usepackage{epsfig}
\usepackage{pst-grad} 
\usepackage{pst-plot} 

\newtheorem{theorem}{Theorem}[section]
\newtheorem*{theorem*}{Theorem}
\newtheorem{corollary}[theorem]{Corollary}
\newtheorem{lemma}[theorem]{Lemma}
\newtheorem*{lemma*}{Lemma}
\newtheorem{proposition}[theorem]{Proposition}
\newtheorem*{proposition*}{Proposition}

\theoremstyle{definition} 
\newtheorem{definition}[theorem]{Definition} 
\newtheorem*{definition*}{Definition} 
\newtheorem{example}[theorem]{Example}
\theoremstyle{remark}
\newtheorem{remark}[theorem]{Remark}

\newcommand{\bi}{\begin{itemize}}

\newcommand{\ei}{\end{itemize}}
\newcommand{\bd}{\begin{description}}
\newcommand{\ed}{\end{description}}

\newcommand{\bqn}{\begin{eqnarray}}
\newcommand{\eqn}{\end{eqnarray}}

\newcommand{\lam}{\lambda}

\newcommand{\g}{\gamma}
\newcommand{\al}{\alpha}
\newcommand{\eps}{\varepsilon}

\newcommand{\R}{\mathbb{R}}
\newcommand{\N}{\mathbb{N}}

\newcommand{\mc}[1]{\mathcal{ #1 }}

\newcommand{\all}{\forall\,}

\newcommand{\la}{\langle}
\newcommand{\ra}{\rangle}

\newcommand{\tx}[1]{\mathrm{#1}}

\newcommand{\wt}[1]{\widetilde{#1}}

\newcommand{\metrp}{g}			

\DeclareMathOperator{\Riccan}{\mathfrak{Ric}}	

\newcommand{\distr}{\mathscr{D}}

\newcommand{\metr}[2]{g(#1,#2)}

\newcommand{\EXP}{\mc{E}}
\newcommand{\Exp}{\mc{E}}

\newcommand{\cc}{c}

\newcommand{\Ric}{\tx{Ric}}
\newcommand{\QQ}{\mc{Q}}
\newcommand{\Qz}{\mc{I}}
\newcommand{\ve}{\mathcal{V}}
\newcommand{\f}{\mathfrak{f}}
\newcommand{\id}{\mathbb{I}} 
\newcommand{\tanf}{\mathsf{T}} 
\newcommand{\dist}{\mathsf{d}}

\newcommand{\Tor}{T}

\newcommand{\Rcan}{\mathfrak{R}} 
\newcommand{\Riem}{\overline{\mathrm{R}}}

\DeclareMathOperator{\rank}{\tx{rank}}
\DeclareMathOperator{\spn}{\tx{span}}
\DeclareMathOperator{\spec}{\tx{spec}}
\DeclareMathOperator{\trace}{\tx{tr}}

\author{Andrei Agrachev$^1$}
\address{$^1$SISSA, Italy, MI RAS and IM SB RAS, Russia}
\email{agrachev@sissa.it}
\author{Davide Barilari$^2$}
\address{$^2$Institut de Math\'ematiques de Jussieu-Paris Rive Gauche, UMR CNRS 7586 - Universit\'e Paris Diderot - Paris 7, Paris, France}
\email{davide.barilari@imj-prg.fr}
\author{Luca Rizzi$^3$}
\address{$^3$CMAP \'Ecole Polytechnique and \'Equipe INRIA GECO Saclay \^Ile-de-France, Paris, France}
\email{luca.rizzi@cmap.polytechnique.fr}


\date{\today}

\title{Sub-Riemannian curvature in contact geometry}

\begin{document}
	
\begin{abstract}
We compare different notions of curvature on contact sub-Riemannian manifolds.
In particular we introduce canonical curvatures as the coefficients of the sub-Riemannian Jacobi equation. The main result is that all these coefficients are encoded in the asymptotic expansion of the horizontal derivatives of the sub-Riemannian distance. We explicitly compute their expressions in terms of the standard tensors of contact geometry. As an application of these results, we obtain a sub-Riemannian version of the Bonnet-Myers theorem that applies to any contact manifold.
\end{abstract}

\maketitle

\tableofcontents

\section{Introduction}

The definition of general curvature-like invariants in sub-Riemannian geometry is a challenging and interesting topic, with many applications to the analysis, topology and geometry of these structures. In the general setting, there is no canonical connection \`a la Levi-Civita and thus the classical construction of the Riemann curvature tensor is not available. Nevertheless, in both the Riemannian and sub-Riemannian setting, the geodesic flow is a well defined Hamiltonian flow on the cotangent bundle: one can then generalize the classical construction of Jacobi fields and define the curvature as the invariants appearing in the Jacobi equation (i.e.\ invariants of the linearization of the geodesic flow).
This approach has been extensively developed in \cite{agrafeedback,geometryjacobi1,lizel,lizel2}. This method, which leads to direct applications, has still some shortcomings since, even if these invariants could be a priori computed via an algorithm, it is extremely difficult to implement.

Another natural approach is to extract geometric invariants from the horizontal derivatives of the sub-Riemannian (squared) distance. 
This approach was developed in \cite{curvature}, where the authors introduce a family of symmetric operators canonically associated with a minimizing trajectory. Under some assumptions, this family admits an asymptotic expansion and each term of this expansion defines a metric invariant. 

The goal of this paper is to revisit  both constructions for contact manifolds, and to  establish a bridge between the two approaches to curvature. The main result is how all invariants of the linearization of the geodesic flow are encoded in the asymptotic expansion  of the squared distance. We explicitly compute these invariants in terms of the standard tensors of contact geometry (Tanno curvature, torsion and Tanno's tensor). Combining these formulas with the results of \cite{BR-comparison}, we obtain a general Bonnet-Myers type result valid for every contact manifold.

\subsection{The contact setting}

Contact manifolds are an important sub-class of corank $1$ sub-Riemannian structures that includes  Yang-Mills type structures, 
Sasakian manifolds, (strongly pseudo-convex) CR structures, and the Heisenberg group.

More precisely, let $M$ be a smooth manifold, with $\dim M = 2d+1$ and $\omega \in \Lambda^{1}M$ be a one-form such that $\omega \wedge (d\omega)^{d}\neq 0$. The \emph{contact distribution} is $\distr:=\ker \omega$. A sub-Riemannian structure on $M$ is given by a smooth scalar product $g$ on $\distr$. In this case we say that $(M,\omega,g)$ is a \emph{contact sub-Riemannian manifold}. This scalar product can be extended to the whole tangent bundle by requiring that the \emph{Reeb vector field} $X_{0}$ is orthogonal to $\distr$ and of norm one. 
 The \emph{contact endomorphism} $J:TM\to TM$ is defined by:
\begin{equation}
g(X,JY)=d\omega(X,Y),\qquad \all X,Y\in \Gamma(TM). 
\end{equation}
By a classical result (see \cite[Thm. 4.4]{blair}) there always exists a choice of the metric $g$ on $\distr$ such that $J^{2}|_{\distr}=-\mathbb{I}$. 
The \emph{Tanno's tensor} is the $(2,1)$ tensor field
\begin{equation}
Q(X,Y):=(\nabla_Y J)X,\qquad \all X,Y \in \Gamma(TM),
\end{equation}
where $\nabla$ is the Tanno connection of the contact manifold (see Sec.~\ref{s:contact} for precise definitions).

\emph{Horizontal curves} (also called \emph{Legendrian} in this setting), are curves $\gamma$ such that $\dot{\gamma}(t) \in \distr_{\gamma(t)}$. The length $\ell(\gamma)$ of an horizontal curve is well defined, and the \emph{sub-Riemannian distance} is
\begin{equation}
\dist(x,y) = \inf\{\ell(\gamma)\mid \gamma \text{ horizontal curve that joins $x$ with $y$}\}.
\end{equation}
This turns a contact sub-Riemannian manifold in a metric space. Geodesics are horizontal curves such that sufficiently short segments realize the distance between their endpoints. They are projections of the integral lines of the geodesic flow on $T^*M$ defined by the Hamiltonian function associated with the (sub-)Riemannian structure. In particular any geodesic is uniquely specified by its \emph{initial covector} $\lambda \in T^*M$. In the Riemannian setting, this is nothing else than the geodesic flow seen on the cotangent space through the canonical isomorphism and the initial covector corresponds to the initial tangent vector of a geodesic. In the sub-Riemannian setting only the ``dual viewpoint'' survives.

\subsection{A family of operators}
For a fixed geodesic $\gamma(t)$, the \emph{geodesic cost function} is
\begin{equation}
c_t(x) :=-\frac{1}{2t}\dist^2(x,\gamma(t)), \qquad x \in M, \,t >0.
\end{equation}
This function is smooth (as a function of $t$ and $x$) for small $t >0$ and $x$ sufficiently close to $x_0=\gamma(0)$. Moreover, its differential recovers the initial covector of the geodesic $\lambda = d_{x_0} c_t$, for all $t>0$. As a consequence, the family of functions $\dot{c}_t:=\frac{d}{dt}c_t$ has a critical point at $x_0$ and its second differential is a well defined quadratic form 
\begin{equation}
d_{x_0}^2 \dot{c}_t : T_{x_0} M \to \R.
\end{equation}
\begin{remark}
The definition of $\dot{c}_t$ makes sense both in the Riemannian and sub-Riemannian setting. In both cases, for any unit-speed geodesic
\begin{equation}
\dot{c}_t(x)=\frac{1}{2} \| \dot\gamma(t) - W^t_{x,\gamma(t)}  \|^2 - \frac{1}{2},
\end{equation}
where $W_{x,\gamma(t)}^t$ is final tangent vector of the geodesic joining $x$ with $\gamma(t)$ in time $t$ (this is uniquely defined precisely where $c_t$ is smooth). See Sec.~\ref{s:interpr}.
\end{remark}
Thanks to the sub-Riemannian metric, the restrictions $d_{x_0}^2 \dot{c}_t|_{\distr_{x_0}}$ define a family of symmetric operators $\QQ_\lam(t):\distr_{x_0} \to \distr_{x_0}$, by the formula
\begin{equation}
d^{2}_{x_0}\dot c_{t}(w)=\metr{ \QQ_\lam(t)w}{w},\qquad   \all t>0, \quad \all w\in \distr_{x_0}.
\end{equation}
The family $\QQ_\lam(t)$ is singular for $t \to 0$, but in this setting we have the following (Theorem~\ref{t:main}).
\begin{theorem} 
The family $t\mapsto t^{2}\QQ_\lam(t)$  can be extended to a smooth family of operators on $\distr_{x_0}$ for small $t\geq 0$, symmetric with respect 	to $g$. Moreover $\Qz_{\lam}:=\displaystyle \lim_{t\to 0^+}t^{2}\QQ_\lam(t) \geq \id > 0.$
\end{theorem}
In particular, we have the Laurent expansion at $t=0$
\begin{equation}\label{eq:mainexp-intro}
\QQ_\lam(t)= \frac{1}{t^{2}}\Qz_{\lam}+ \sum_{i=0}^{m}\QQ^{(i)}_\lam t^{i}+O(t^{m+1}).
\end{equation}
Every operator $\QQ^{(i)}_\lambda  : \distr_{x_0} \to \distr_{x_0}$, for $i \in \N$, is an invariant of the metric structure and, in this sense, $\QQ_\lam(t)$ can be thought of as a \emph{generating function} for metric invariants. These operators, together with $\Qz_\lam$, contain all the information on the germ of the structure along the fixed geodesic (clearly any modification of the structure that coincides with the given one on a neighborhood of the geodesic gives the same family of operators).
\begin{remark} 
Applying this construction in the case of a Riemannian manifold,  where $\distr_{x_0} = T_{x_0}M$, one finds that $\Qz_\lam = \mathbb{I}$ for any geodesic, and the operator $\QQ^{(0)}_\lam$ is the ``directional'' sectional curvature in the direction of the geodesic:
\begin{equation}
\QQ^{(0)}_\lam=\frac13 \Riem(\cdot,\dot\gamma)\dot\gamma,
\end{equation}
where $\dot\gamma$ is the initial vector of the geodesic and $\Riem$ is the curvature of the Levi-Civita connection.
\end{remark}

A similar construction has been carried out, in full generality, for the geometric structures arising from affine optimal control problems with Tonelli-type Lagrangian (see \cite{curvature}). In that setting the geodesic is replaced by a minimizer of the optimization problem. Under generic assumptions on the extremal, the singularity of the family $\QQ_\lam(t)$ is controlled. 

In the contact setting, due to the absence of abnormal minimizers, these assumptions are always satisfied for any non-trivial geodesic. One of the purposes of this paper is to provide a simpler and direct proof of the existence of the asymptotic.

\subsection{The singular term}
A first surprise is that, already in the contact case, $\Qz_\lam$ is a non-trivial invariant. We obtain the complete characterization of the singular term of~\eqref{eq:mainexp-intro} (Theorem~\ref{t:main2}).
\begin{theorem} 
The symmetric operator $\Qz_{\lam}: \distr_{x_{0}}\to \distr_{x_{0}}$ satisfies
\begin{itemize}
\item[(i)] $\spec\Qz_{\lam}=\{1,4\}$,
\item[(ii)] $\trace\Qz_{\lam}=2d+3$.  
\end{itemize}
More precisely, let $K_{\gamma(0)} \subset \distr_{x_0}$ be the hyperplane $d\omega$-orthogonal to $\dot\gamma(0)$, that is 
\begin{equation}
K_{\gamma(0)}:=\{v \in \distr_{\gamma(0)}\mid d\omega(v,\dot\gamma(0)) = 0\}.
\end{equation}
Then $K_{\gamma(0)}$ is the eigenspace corresponding to eigenvalue $1$ (and geometric multiplicity $2d-1$) and $K_{\gamma(0)}^\perp\cap \distr_{x_0}$ is the eigenspace corresponding to eigenvalue $4$ (and geometric multiplicity $1$).
\end{theorem}
$\Qz_\lam$ is a structural invariant that does not depend on the metric, but only on the fact that the distance function comes from a contact sub-Riemannian structure (operators $\Qz_\lam$ coming from different metrics $g$ on the same contact distribution have the same spectral invariants).
\begin{remark}
The trace $\mathcal{N} := \trace\Qz_\lam = 2d+3$, coincides with the \emph{geodesic dimension} (of the contact sub-Riemannian structure) defined in \cite[Sec. 5.6]{curvature} (see also \cite{R-MCP} for general metric measure spaces). If $\Omega_t$ is the geodesic homothety with center $x_0$ and ratio $t \in [0,1]$ of a measurable, bounded set $\Omega \subset M$ with positive measure, we have the following \emph{asymptotic measure contraction property}:
\[
\lim_{t \to 0} \frac{\log \mu(\Omega_t)}{\log t} = \mathcal{N},
\]
for any smooth measure $\mu$. i.e. $\mathcal{N}$ is the order of vanishing of $\mu(\Omega_t)$ for $t \to 0$.
\end{remark}
The regular terms of the asymptotic~\ref{eq:mainexp-intro} are \emph{curvature-like} invariants. The next main result is the explicit relation of the operators $\QQ^{(i)}_\lambda$ with the symplectic invariants of the linearization of the geodesic flow, that we introduce now.

\subsection{Linearized Hamiltonian flow}

In the (sub-)Riemannian setting, the geodesic flow $\phi_t :T^*M \to T^*M$ is generated by the Hamiltonian function $H \in C^\infty(T^*M)$ (the co-metric of the sub-Riemannian structure). More precisely, if $\sigma$ is the canonical symplectic structure on $T^*M$, then the Hamiltonian vector field $\vec{H}$ is defined by $\sigma(\cdot,\vec{H}) = dH$, and $\phi_t = e^{t\vec{H}}$. Integral lines $\lambda(t) = e^{t\vec{H}}(\lambda)$ of the geodesic flow are usually called \emph{extremals}. Geodesics are then projections of non-trivial extremals $\gamma(t) = e^{t\vec{H}}(\lambda)$ (non-trivial is equivalent to $H(\lambda)\neq 0$). For any fixed extremal, and  initial datum $\xi \in T_{\lambda}(T^*M)$, we define the vector field along the extremal
\begin{equation}
X_\xi(t) := e^{t\vec{H}}_*\xi \in T_{\lambda(t)}(T^*M).
\end{equation}
The set of these vector fields is a $2n$-dimensional vector space that coincides with the space of solutions of the \emph{(sub-)Riemannian Jacobi equation}
\begin{equation}
\dot{X} = 0,
\end{equation}
where $\dot{X}:=\mathcal{L}_{\vec{H}}X$ denotes the Lie derivative in the direction of $\vec{H}$. 

Pick a Darboux frame $\{E_i(t),F_i(t)\}_{i=1}^{n}$ along $\lambda(t)$ (to fix ideas, one can think at the canonical basis $\{\partial_{p_i}|_{\lambda(t)},\partial_{q_i}|_{\lambda(t)}\}$ induced by a choice of coordinates $(q_1,\ldots,q_n)$ on $M$). In terms of this frame, the Jacobi field $X(t)$ has components $(p(t),q(t)) \in \R^{2n}$:
\begin{equation}
X(t) = \sum_{i=1}^n p_{i}(t) E_{i}(t) + q_{i}(t) F_{i}(t).
\end{equation}
In the Riemannian case, one can choose a \emph{canonical Darboux frame} (satisfying special equations, related with parallel transport) such that the components $(p(t),q(t)) \in \R^{2n}$ satisfy
\begin{equation}\label{eq:jacobicoord-intro}
\dot{p} = - R(t) q, \qquad \dot{q} = p ,
\end{equation}
for some smooth family of symmetric matrices $R(t)$. It turns out this class of frames is defined up to a constant orthogonal transformation and $R(t)$ is the matrix representing the \emph{curvature operator} $\Rcan_{\gamma(t)} : T_{\gamma(t)}M \to T_{\gamma(t)}M$ in the direction of the geodesic, in terms of a parallel transported frame. In particular
\begin{equation}
\Rcan_{\gamma(t)}v = \Riem(v,\dot{\gamma}(t))\dot{\gamma}(t), \qquad \all v \in T_{\gamma(t)} M.
\end{equation}
where $\Riem$ is the Riemannian curvature tensor. From Eq.~\eqref{eq:jacobicoord-intro}, it follows the classical classical Jacobi equation written in terms of a parallel transported frame: $\ddot{q} + R(t) q = 0$. In this language, the Riemann curvature arises as a set of invariants of the linearization of the geodesic flow.

\subsubsection{Canonical curvatures}
Analogously, in the sub-Riemannian setting, we might look for a \emph{canonical Darboux frame} such that the  Jacobi has, in coordinates, the simplest possible form. This analysis begun in \cite{agrafeedback} and has been completed in a very general setting in \cite{lizel}.  The ``normal form'' of the Jacobi equation defines a series of invariants of the sub-Riemannian structure along the given geodesic:
\begin{itemize}
\item A \emph{canonical splitting} of the tangent space along the geodesic:
\begin{equation}
 T_{\gamma(t)} M =\bigoplus_{\alpha} S^\alpha_{\gamma(t)},
\end{equation}
where $\alpha$ runs over a set of indices that depends on the germ of the sub-Riemannian structure along the geodesic.
\item A \emph{canonical curvature operator}  $\Rcan_{\gamma(t)}: T_{\gamma(t)} M \to T_{\gamma(t)} M$.
\end{itemize}
The operator $\Rcan_{\gamma(t)}$ (and its partial traces) is the correct object to bound to obtain sectional-type (and Ricci-type) comparison theorems in the general sub-Riemannian setting, as it controls the evolution of the Jacobian of the exponential map (see for instance \cite{BR-comparison}).

The general formulation is complicated, since the very structure of the normal form depends on the type of sub-Riemannian structure. In Section~\ref{s:Jac} we give an ad-hoc presentation for the contact case. 
In particular, we prove the following (Theorem~\ref{t:cansplitting}).
\begin{theorem}[Canonical splitting]
Let $\gamma(t)$ be a unit speed geodesic of a contact sub-Rieman\-nian structure with initial covector $\lambda$. Then the canonical splitting is given by
\begin{align}
S^a_{\gamma(t)}& := \spn\{X_0 - 2Q(\dot\gamma,\dot\gamma)-h_0\dot\gamma\},  & \dim S^a_{\gamma(t)} & = 1, \\
S^b_{\gamma(t)}& := \spn\{J\dot\gamma \},   & \dim S^b_{\gamma(t)} & = 1, \\
S^c_{\gamma(t)}& := J\dot\gamma^\perp \cap \distr_{\gamma(t)},   & \dim S^c_{\gamma(t)} & = 2d-1,
\end{align}
where $X_0$ is the Reeb field, $Q$ is the Tanno tensor, $h_0 = \langle \lambda,X_0\rangle$ and everything is computed along the  extremal. Indeed $\distr_{\gamma(t)} = S^b_{\gamma(t)} \oplus S^c_{\gamma(t)}$ and $\dot\gamma(t) \in S^c_{\gamma(t)}$.
\end{theorem}
The above theorem follows from the explicit computation of the canonical frame. Moreover, we obtain an explicit expression for $\Rcan_{\gamma(t)}$ (Theorems~\ref{t:curvQ=0}-\ref{t:curvQneq0}).

\subsection{Relation between the two  approaches}
The main goal of this paper is to find the relation between the curvature-like objects introduced so far: on one hand, the canonical curvature operator $\Rcan_{\gamma(t)}$ that we have just defined, on the other hand, the invariants $\QQ^{(i)}_\lambda$, for $i \geq 0$, defined by the asymptotics~\eqref{eq:mainexp-intro} (both associated with a given geodesic with initial covector $\lambda$). 
 
 Notice that $\Rcan$ is a canonical operator defined on a space of dimension $n$, while the operators $\QQ^{(i)}$ are defined on a space of dimension $k$, equal to the dimension of the distribution. In the Riemannian case $k=n$, and a dimensional argument suggests that the first element $\QQ^{(0)}$ should ``contain'' all the canonical curvatures. Indeed we have
\begin{equation}
\Rcan_{\gamma(t)} = 3\QQ^{(0)}_{\lambda(t)} = \Riem(\cdot,\dot\gamma(t))\dot\gamma(t).
\end{equation}
In the sub-Riemannian setting the relation is much more complicated and in general the first element $\QQ^{(0)}$ recovers only \emph{a part} of the canonical curvature $\Rcan$. More precisely, as proved in \cite{curvature} 
\begin{equation}
\Rcan_{\gamma(t)}\big|_{\distr_{\g(t)}} = 3\QQ^{(0)}_{\lambda(t)}. 
\end{equation}
It turns out that, on  contact structures, we recover the whole $\Rcan$ by computing the higher order invariants $\QQ^{(0)}$,  $\QQ^{(1)}$ and $\QQ^{(2)}$ (see Theorem~\ref{t:main3} for the  explicit relations).

\subsection{Comparison theorems}

The restriction of the curvature operator on the invariant subspaces $S^\alpha_{\gamma(t)}$ is denoted $\Rcan_{\gamma(t)}^{\alpha\beta}$, for $\alpha,\beta =a,b,c$. We have more than one Ricci curvature, one partial trace for each subspace
\begin{equation}
\Riccan^\alpha_{\gamma(t)}:= \trace \left( \Rcan_{\gamma(t)}^{\alpha\alpha} : S_{\gamma(t)}^\alpha \to S_{\gamma(t)}^\alpha \right), \qquad \alpha = a,b,c.
\end{equation}
In the Riemannian case, we only have one subspace (the whole tangent space) and only one average: the classical Ricci curvature.

In \cite{BR-comparison}, under suitable conditions on the canonical curvature of a given sub-Riemannian geodesic, we obtained bounds on the first conjugate time along the geodesic and, in particular, Bonnet-Myers type results. In Sec.~\ref{s:bm} we first apply the results of \cite{BR-comparison} to contact structures. It is interesting to express these conditions in terms of the classical tensors of the contact structure (Tanno's tensor, curvature and torsion). With the explicit expressions for the canonical curvature of Sec.~\ref{s:computations}, we obtain the following results (Theorem~\ref{t:bmym}).

\begin{theorem} \label{t:spacca}
Consider a complete, contact structure of dimension $2d+1$, with $d>1$. Assume that there exists constants $\kappa_{1} > \kappa_{2} \geq 0$ such that, for any horizontal unit vector $X$
\begin{equation}\label{eq:bmYM-intro}
\Ric(X) - R(X,JX,JX,X) \geq (2d-2)\kappa_{1},\qquad \|Q(X,X)\|^{2}\leq (2d-2)\kappa_{2}.
\end{equation}
Then the manifold is compact with sub-Riemannian diameter not greater than $\pi/\sqrt{\kappa_{1}-\kappa_{2}}$, and the fundamental group is finite.
\end{theorem}
We stress that the curvatures appearing in~\eqref{eq:bmYM-intro} are computed w.r.t.\ Tanno connection. This generalizes the results for Sasakian structures obtained in \cite{LL-BishopLaplacian,LLZ-Sasakian}.

Finally, we obtain the following corollary for (strongly pseudo-convex) CR manifolds (that is, for $Q = 0$). Notice that this condition is strictly weaker than Sasakian. Observe that in the CR case, Tanno's curvature coincides with the classical Tanaka-Webster curvature.
\begin{corollary}
Consider a complete, (strongly pseudo-convex) CR structure of dimension $2d+1$, with $d>1$, such that, for any horizontal unit vector $X$
\begin{equation}
\Ric(X) - R(X,JX,JX,X) \geq (2d-2)\kappa > 0.
\end{equation}
Then the manifold is compact with sub-Riemannian diameter not greater than $\pi/\sqrt{\kappa}$, and the fundamental group is finite.
\end{corollary}

Besides the above references, other Bonnet-Myers type results are found in the literature, proved with different techniques and for different sub-Riemannian structures. For example, with heat semigroup approaches: for Yang-Mills type structures with transverse symmetries  \cite{garofalob} and Riemannian foliations with totally geodesic leaves \cite{BKW-weitzenbock}. With direct computation of the second variation formula: for 3D contact CR structures \cite{Rumin} and for general 3D contact ones \cite{Hughen-PhD}. Finally, with Riccati comparison techniques for 3D contact \cite{AAPL}, 3-Sasakian \cite{RS-3-Sasakian} and for any sub-Riemannian structure \cite{BR-comparison}.

A compactness result for contact structures is also obtained  \cite{baudoincontact} by applying the classical Bonnet-Myers theorem to a suitable Riemannian extension of the metric.

\subsection{Final comments and open questions}

In the contact setting, the invariants $\QQ^{(i)}$ for $i=0,1,2$ recover the whole canonical curvature operator $\Rcan$. It is natural to conjecture that, in the general case, there exists $N\in \N$  (depending on the sub-Riemannian structure) such that the invariants  $\QQ^{(i)}$ for $i=0,\ldots, N,$ recover the whole canonical curvature $\Rcan$.  Already in the contact case the relation is complicated due to the high number of derivatives required.

Finally, the comparison results obtained above  (and in the general sub-Riemannian setting in \cite{BR-comparison}) rely on the explicit computations of $\Rcan$ and its traces. In view of the relation we obtained between $\Rcan$ and the operators $\QQ^{(i)}$, it is natural to ask whether it is possible to obtain comparison theorems in terms of suitable $C^{N}$-bounds on the geodesic cost (for some finite $N$ depending on the sub-Riemannian structure).

\subsection*{Acknowledgements}
The first author has been supported by the grant of the Russian Federation for the state support of
research, Agreement No 14 B25 31 0029. This research has also been supported  by the European Research Council, ERC StG 2009 ``GeCoMethods,'' contract number 239748 and by the iCODE institute, research project of the Idex Paris-Saclay. The second and third authors were supported by the Grant ANR-15-CE40-0018 of the ANR. The third author was supported by the SMAI project ``BOUM''. This research, benefited from the support of the ``FMJH Program Gaspard Monge in optimization and operation research'', and from the support to this program from EDF.

We thank the Chinese University of Hong Kong, where part of this project has been carried out. We are grateful to Paul W.Y. Lee for his kind invitation and his precious contribution. We thank Fabrice Baudoin for useful discussions on Yang-Mills structures. We also warmly thank Stefan Ivanov for pointing  out a redundant assumption in a previous version of Theorem \ref{t:spacca}.

\section{Preliminaries} \label{s:srg}

We recall some basic facts in sub-Riemannian geometry. We refer to \cite{nostrolibro} for further details. 

Let $M$ be a smooth, connected manifold of dimension $n \geq 3$. A sub-Riemannian structure on $M$ is a pair $(\distr,\metrp)$ where $\distr$ is a smooth vector distribution of constant rank $k\leq n$ satisfying the \emph{H\"ormander condition} (i.e. $\mathrm{Lie}_x\distr = T_x M$, $\forall x \in M$) and $\metrp$ is a smooth Riemannian metric on $\distr$. A Lipschitz continuous curve $\g:[0,T]\to M$ is \emph{horizontal} (or \emph{admissible}) if $\dot\g(t)\in\distr_{\g(t)}$ for a.e. $t \in [0,T]$.
Given a horizontal curve $\g:[0,T]\to M$, the \emph{length of $\g$} is
\begin{equation}
\ell(\g)=\int_0^T \|\dot{\g}(t)\|dt,
\end{equation}
where $\|\cdot\|$ is the norm induced by $\metrp$. The \emph{sub-Riemannian distance} is the function
\begin{equation}
\dist(x,y):=\inf \{\ell(\g)\mid \g(0)=x,\g(T)=y, \g\; \text{horizontal}\}.
\end{equation}
The Rashevsky-Chow theorem (see \cite{chow,rashevsky}) guarantees the finiteness and the continuity of $\dist :M\times M \to \R$ with respect to the topology of $M$. The space of vector fields (resp. horizontal vector fields) on $M$ is denoted by $\Gamma(TM)$ (resp. $\Gamma(\distr)$).

Locally, the pair $(\distr,\metrp)$ can be given by assigning a set of $k$ smooth vector fields that span $\distr$, orthonormal for $\metrp$. In this case, the set $\{X_1,\ldots,X_k\}$ is called a \emph{local orthonormal frame} for the sub-Riemannian structure. 

A sub-Riemannian \emph{geodesic} is an admissible curve $\g:[0,T]\to M$ such that $\|\dot\g(t)\|$ is constant and for every $t \in [0,T]$ there exists an interval $t_1< t < t_2$ such that the restriction $\g|_{[t_1,t_2]}$ realizes the distance between its endpoints. The length of a geodesic is invariant by reparametrization of the latter. Geodesics with $\|\dot\g(t)\| = 1$ are called \emph{length parametrized} (or of \emph{unit speed}). A sub-Riemannian manifold is \emph{complete} if $(M,\dist)$ is complete as a metric space.

With a sub-Riemannian structure we associate the Hamiltonian function $H \in C^\infty(T^*M)$
\begin{equation}
H(\lambda) = \frac{1}{2}\sum_{i=1}^k\la \lambda, X_i\ra^2, \qquad \forall \lambda \in T^*M,
\end{equation}
for any local orthonormal frame $X_1,\ldots,X_k$, where $\la\lambda,\cdot\ra$ denotes the action of the covector $\lambda$ on vectors.  Let $\sigma$ be the canonical symplectic form on $T^*M$. 
For any function $a \in C^\infty(T^*M)$, the associated Hamiltonian vector field $\vec{a}$ is defined by the formula $da = \sigma(\cdot,\vec{a})$. For $i=1,\ldots,k$ let $h_i \in C^\infty(T^*M)$ be the linear-on-fibers functions $h_i(\lambda):= \la \lambda,X_i\ra$. Clearly,
\begin{equation}
H = \frac{1}{2}\sum_{i=1}^k h_i^2, \qquad\text{and}\qquad \vec{H} = \sum_{i=1}^k h_i \vec{h}_i.
\end{equation}

\subsection{Contact sub-Riemannian structures} In this paper we consider contact structures defined as follows. Let $M$ be a smooth manifold of odd dimension $\dim M = 2d+1$, and let $\omega \in \Lambda^{1}M$ a one-form such that $\omega \wedge (d\omega)^{d}\neq 0$. In particular  $\omega \wedge (d\omega)^{d}$ defines a volume form and $M$ is orientable. The \emph{contact distribution} is $\distr:=\ker \omega$. Indeed $d\omega|_{\distr} $ is non-degenerate.
The choice of a scalar product $g$ on $\distr$ defines a sub-Riemannian structure on $M$. In this case we say that $(M,\omega,g)$ is a \emph{contact sub-Riemannian manifold}.

Trajectories minimizing the distance between two points are solutions of first-order necessary conditions for optimality, given by a weak version of the Pontryagin Maximum Principle (see \cite{pontrybook}, or  \cite{nostrolibro} for an elementary proof). We denote by $\pi:T^*M \to M$ the standard bundle projection.
\begin{theorem}\label{t:pmpw}
Let $M$ be a contact sub-Riemannian manifold and let $\gamma:[0,T] \to M$ be a sub-Riemannian geodesic. Then there exists a Lipschitz curve $\lambda: [0,T] \to T^*M$, such that $\pi \circ \lambda = \gamma$ and for all $t \in [0,T]$:
\begin{equation} \label{eq:nnn}
\dot\lambda(t) = \vec{H}(\lambda(t)).
\end{equation}
\end{theorem}
If $\lambda:[0,T] \to M$ is a curve satisfying \eqref{eq:nnn}, it is called a \emph{normal} \emph{extremal}. It is well known that if $\lambda(t)$ is a normal extremal, then $\lambda(t)$ is smooth and its projection $\gamma(t):=\pi(\lambda(t))$ is a smooth geodesic. Let us recall that this characterization is not complete on a general sub-Riemannian manifold, since also abnormal extremals can appear.

Let $\lambda(t)=e^{t\vec{H}}(\lambda_0)$ be the integral curve of the Hamiltonian vector field $\vec{H}$ starting from $\lambda_0$. The sub-Riemannian \emph{exponential map} (from $x_{0}$) is
\begin{equation}\label{eq:expmap}
\EXP_{x_{0}}: T_{x_0}^*M \to M, \qquad \EXP_{x_{0}}(\lambda_{0}):= \pi(e^{\vec{H}}(\lambda_{0})).
\end{equation}
Unit speed geodesics correspond to initial covectors such that $H(\lambda_0) = 1/2$.

\begin{definition}
Let $\gamma(t)=\Exp_{x_{0}}(t\lam_{0})$ be a normal unit speed geodesic. We say that $t > 0$ is a \emph{conjugate time} (along the geodesic) if $t\lambda_0$ is a critical point of $\Exp_{x_0}$. The \emph{first conjugate time} is  $t_*(\lam_{0})=\inf\{t>0\mid  t\lam_{0}\  \text{is a critical point of}\  \Exp_{x_{0}}\}$.
\end{definition}

On a contact sub-Riemannian manifold or, in general, when there are no non-trivial abnormal extremals, the first conjugate time is separated from zero (see for instance \cite{nostrolibro,agrexp}) and, after its first conjugate time, geodesics lose local optimality. 

\subsection{Contact geometry}\label{s:contact}
Given a contact manifold $(M,\omega)$, the \emph{Reeb vector field} $X_{0}$ is the unique vector field satisfying $\omega(X_{0})=1$ and $d\omega(X_{0},\cdot)=0$. Clearly $X_0$ is transverse to $\distr$. We can extend the sub-Riemannian metric $g$ on $\distr$ to a global Riemannian structure (that we denote with the same symbol $g$) by promoting $X_0$ to an unit vector orthogonal to $\distr$.

We define the \emph{contact endomorphism} $J:TM\to TM$ by:
\begin{equation}
g(X,JY)=d\omega(X,Y),\qquad \all X,Y\in \Gamma(TM). 
\end{equation}
Clearly $J$ is skew-symmetric w.r.t.\ to $g$. By a classical result (see \cite[Thm. 4.4]{blair}) there always exists a choice of the metric $g$ on $\distr$ such that $J^{2}=-\mathbb{I}$ on $\distr$ and $J(X_{0})=0$, or equivalently
\begin{equation}\label{eq:compa}
J^{2}=-\mathbb{I}+\omega \otimes X_{0}.
\end{equation}
In this case, $g$ is said to be \emph{compatible} with the contact structure and $(M,\omega,g,J)$ is usually referred to as a \emph{contact metric structure} or a \emph{contact Riemannian structure}. In this paper, we always assume the metric $g$ to be compatible.

\begin{theorem}[Tanno connection, \cite{blair,tanno89}]
There exists a unique linear connection $\nabla$  on $(M,\omega,g,J)$ such that
\begin{itemize}
\item[(i)] $\nabla\omega = 0$,
\item[(ii)] $\nabla X_0 = 0$,
\item[(iii)] $\nabla g = 0$,
\item[(iv)] $\Tor(X,Y) = d\omega(X,Y) X_0$ for any $X,Y \in \Gamma(\distr)$,
\item[(v)] $\Tor(X_0,JX) = -J \Tor(X_0,X)$ for any vector field $X \in \Gamma(TM)$,
\end{itemize}
where $\Tor$ is the torsion tensor of $\nabla$.
\end{theorem}
The \emph{Tanno's tensor} is the $(2,1)$ tensor field defined by
\begin{equation}
Q(X,Y):=(\nabla_Y J)X,\qquad \all X,Y \in \Gamma(TM).
\end{equation}
A fundamental result due to Tanno is that $(M,\omega,g,J)$ is a (strongly pseudo-convex) CR manifold if and only if $Q=0$. In this case, Tanno connection is Tanaka-Webster connection. Thus, Tanno connection is a natural generalisation of the Tanaka-Webster connection for contact structures that are not CR.

\subsection{K-type structures}
If $X$ is an horizontal vector field, so is $\Tor(X_0,X)$. As a consequence, if we define $\tau(X) = \Tor(X_0,X)$, $\tau$ is a symmetric horizontal endomorphism which satisfies $\tau \circ J + J \circ \tau = 0$, by property (v). On CR manifolds, $\tau$ is called \emph{pseudo-Hermitian torsion}. 
A contact structure is $K$-type iff $X_0$ is a Killing vector field or, equivalently, if $\tau =0$.

\subsection{Yang-Mills structures}
We say that a contact sub-Riemannian structure is \emph{Yang-Mills} if the torsion $T$ of Tanno connection satisfies 
\begin{equation}\label{eq:YMdef}
\sum_{i=1}^{2d}(\nabla_{X_{i}}T)(X_{i},Y)=0,\qquad \all Y\in \Gamma(TM),
\end{equation}
for every orthonormal frame $X_{1},\ldots,X_{2d}$ of $\distr$. This definition  coincides with the classical one given in \cite[Sec. 5.1]{Falcitelli} for totally geodesic foliations with bundle like metric, and generalizes it to contact sub-Riemannian structures that are not foliations, e.g. when $\tau \neq 0$. 

\subsection{Sasakian structures}
If $(M,\omega,g,J)$ is a contact sub-Riemannian manifold with Reeb vector $X_0$, consider the manifold $M\times\R$. We denote vector fields on $M\times \R$ by $(X,f\partial_t)$, where $X$ is tangent to $M$ and $t$ is the coordinate on $\R$. Define the $(1,1)$ tensor
\begin{equation}
\mathbf{J}(X,f\partial_t) = (JX-fX_0,\omega(X)\partial_t).
\end{equation}
Indeed $\mathbf{J}^2 = -\mathbb{I}$ and thus defines an almost complex structure on $M\times \R$ (this clearly was not possible on the odd-dimensional manifold $M$). We say that the contact sub-Riemannian structure $(M,\omega,g,J)$ is \emph{Sasakian} (or \emph{normal}) if the almost complex structure $\textbf{J}$ comes  from a true complex structure. A celebrated theorem by Newlander and Nirenberg states that this condition is equivalent to the vanishing of the Nijenhuis tensor of $\mathbf{J}$. For a $(1,1)$ tensor $T$, its Nijenhuis tensor $[T,T]$ is the $(2,1)$ tensor
\begin{equation}
[T,T](X,Y) := T^2[X,Y] + [TX,TY] - T[TX,Y] - T[X,T Y].
\end{equation}
In terms of the original structure, the integrability condition $[\mathbf{J},\mathbf{J}] =0$ is equivalent to
\begin{equation}\label{eq:integr}
[J,J](X,Y) + d\omega(X,Y) X_0 = 0.
\end{equation}
\begin{remark}
A contact sub-Riemannian structure is Sasakian if and only if $Q=0$ and $\tau =0$, that is if and only if it is CR and $K$-type (see \cite[Theorems 6.7, 6.3]{blair}). Notice that a three-dimensional contact sub-Riemannian structure is automatically CR (see \cite[Cor. 6.4]{blair}), in particular, it is Sasakian if and only if it is $K$-type.
\end{remark}

\section{Jacobi fields revisited}\label{s:Jac}

Let $\lambda \in T^*M$ be the initial covector of a geodesic, projection of the extremal $\lambda(t) = e^{t\vec{H}}(\lambda)$. For any $\xi \in T_\lambda(T^*M)$ we define the field along the extremal $\lambda(t)$ as 
\begin{equation}
X(t):= e^{t\vec{H}}_* \xi \in T_{\lambda(t)}(T^*M).
\end{equation}
\begin{definition}
The set of vector fields obtained in this way is a $2n$-dimensional vector space, that we call \emph{the space of Jacobi fields along the extremal}.
\end{definition}
In the Riemannian case, the projection $\pi_*$ is an isomorphisms between the space of Jacobi fields along the extremal and the classical space of Jacobi fields along the geodesic $\gamma$. Thus, this definition agrees with the standard one in Riemannian geometry and does not need curvature or connection.

In Riemannian geometry, the subspace of Jacobi fields vanishing at zero carries information about conjugate points along the given geodesic. This corresponds to the subspace of Jacobi fields along the extremal such that $\pi_* X(0) = 0$. This motivates the following construction.

For any $\lambda\in T^*M$, let $\ve_{\lambda}:= \ker \pi_*|_\lambda \subset T_{\lambda}(T^*M)$ be the \emph{vertical subspace}. We define the family of Lagrangian subspaces along the extremal
\begin{equation}
\mc{L}(t):= e^{t\vec{H}}_* \ve_\lambda \subset T_{\lambda(t)}(T^*M).
\end{equation}

\begin{remark}
A time $t>0$ is a conjugate time along $\gamma$ if $\mc{L}(t) \cap \ve_{\lambda(t)} \neq \{0\}$. The first conjugate time is the smallest conjugate time, namely $t_*(\gamma) = \inf\{t>0 \mid \mc{L}(t) \cap \ve_{\lambda(t)} \neq \{0\}\}$.
\end{remark}

Notice that conjugate points correspond to the critical values of the sub-Riemannian exponential map with base at $\gamma(0)$. In other words, if $\gamma(t)$ is conjugate with $\gamma(0)$ along $\gamma$, there exists a one-parameter family of geodesics starting at $\gamma(0)$ and ending at $\gamma(t)$ at first order. Indeed, let $\xi \in \ve_\lambda$ such that $\pi_* \circ e^{t\vec{H}}_* \xi = 0$, then the vector field $\tau \mapsto \pi_* \circ e^{\tau\vec{H}}_* \xi$  is precisely the vector field along $\gamma(\tau)$ of the aforementioned variation. 

\subsection{Linearized Hamiltonian}
For any vector field $X(t)$ along an integral line $\lambda(t)$ of the (sub-)Riemannian Hamiltonian flow, a dot denotes the Lie derivative in the direction of $\vec{H}$:
\begin{equation}
\dot{X}(t) := \left.\frac{d}{d\eps}\right|_{\eps=0} e^{-\eps \vec{H}}_* X(t+\eps).
\end{equation}
The space of Jacobi fields along the extremal $\lambda(t)$ coincides with the set of solutions of the \emph{(sub-)Riemannian Jacobi equation}:
\begin{equation}
\dot{X} = 0.
\end{equation}
We want to write the latter in a more standard way. Pick a Darboux frame $\{E_i(t),F_i(t)\}_{i=1}^{n}$ along $\lambda(t)$ (to fix ideas, one can think at the canonical basis $\{\partial_{p_i}|_{\lambda(t)},\partial_{q_i}|_{\lambda(t)}\}$ induced by a choice of coordinates $(q_1,\ldots,q_n)$ on $M$). In terms of this frame, $X(t)$ has components $(p(t),q(t)) \in \R^{2n}$:
\begin{equation}
X(t) = \sum_{i=1}^n p_{i}(t) E_{i}(t) + q_{i}(t) F_{i}(t).
\end{equation}
The elements of the frame satisfy
\begin{equation}\label{eq:Jacobiframe}
\begin{pmatrix}
\dot{E} \\
\dot{F}
\end{pmatrix} = 
\begin{pmatrix}
C_1(t) & -C_2(t) \\
R(t) & -C_1^*(t)
\end{pmatrix} \begin{pmatrix}
E\\
F
\end{pmatrix},
\end{equation}
for some smooth families of $n\times n$ matrices $C_1(t),C_2(t),R(t)$, where $C_2(t) = C_2(t)^*$ and $R(t)= R(t)^*$. The notation for these matrices will be clear in the following. We only stress here that the particular structure of the equations is implied solely by the fact that the frame is Darboux. Moreover, $C_2(t) \geq 0$ as a consequence of the non-negativity of the sub-Riemannian Hamiltonian (in the Riemannian case $C_2(t) > 0$). In turn, the Jacobi equation, written in terms of the components $p(t),q(t)$, becomes
\begin{equation}
\begin{pmatrix}\label{eq:Jacobicoord}
\dot{p} \\ \dot{q}
\end{pmatrix} = \begin{pmatrix} - C_1(t)^* & -R(t) \\ C_2(t) & C_1(t)
\end{pmatrix} \begin{pmatrix}
p \\ q
\end{pmatrix}.
\end{equation}

\subsection{Canonical frame: Riemannian case}
In the Riemannian case one can choose a suitable frame (related with parallel transport) in such a way that, in Eq.~\eqref{eq:Jacobicoord}, $C_1(t) =0$, $C_2(t) = \mathbb{I}$ (in particular, they are constant), and the only remaining non-trivial block $R(t)$ is the curvature operator along the geodesic. The precise statement is as follows (see \cite{BR-comparison}).
\begin{proposition}\label{p:riemcan}
Let $\lambda(t)$ be an integral line of a Riemannian Hamiltonian. There exists a smooth moving frame $\{E_i(t),F_i(t)\}_{i=1}^n$ along $\lambda(t)$ such that:
\begin{itemize}
\item[(i)] $\spn\{E_1(t),\ldots,E_n(t)\} = \ve_{\lambda(t)}$.
\item[(ii)] It is a Darboux basis, namely
\begin{equation}
\sigma(E_i,E_j) = \sigma(F_i,F_j) = \sigma(E_i,F_j) - \delta_{ij} = 0, \qquad i,j=1,\ldots,n.
\end{equation}
\item[(iii)] The frame satisfies the structural equations
\begin{equation}
\dot{E}_i = - F_i, \qquad \dot{F}_i = \sum_{j=1}^n R_{ij}(t) E_j,
\end{equation}
for some smooth family of $n\times n$ symmetric matrices $R(t)$.
\end{itemize}
Moreover, the projections $f_i(t):=\pi_*F_i(t)$ are a parallel transported orthonormal frame along the geodesic $\gamma(t)$. 

Properties (i)-(iii) uniquely define the moving frame up to orthogonal transformations: if $\{\wt{E}_i(t),\wt{F}_j(t)\}_{i=1}^n$ is another frame satisfying (i)-(iii), for some family $\wt{R}(t)$, then there exists a constant $n\times n$ orthogonal matrix $O$ such that 
\begin{equation}\label{eq:orthonormal}
\wt{E}_i(t) = \sum_{j=1}^n O_{ij}E_j(t), \qquad  \wt{F}_i(t) = \sum_{j=1}^nO_{ij}F_j(t), \qquad \wt{R}(t) = O R(t) O^*. 
\end{equation}
\end{proposition}
A few remarks are in order. Property (ii) implies that $\spn\{E_1,\ldots,E_n\}$, $\spn\{F_1,\ldots,F_n\}$, evaluated at $\lambda(t)$, are Lagrangian subspaces of $T_{\lambda(t)}(T^*M)$. Eq.~\eqref{eq:orthonormal} reflects the fact that a parallel transported frame is defined up to a constant orthogonal transformation. In particular, one can use properties (i)-(iii) to \emph{define} a parallel transported frame $\gamma(t)$ by $f_i(t):=\pi_* F_i(t)$. The symmetric matrix $R(t)$ induces a well defined operator $\Rcan_{\gamma(t)}:T_{\gamma(t)}M \to T_{\gamma(t)}M$
\begin{equation}\label{eq:defff}
\Rcan_{\gamma(t)} f_i(t) := \sum_{j=1}^n R_{ij}(t) f_j(t). 
\end{equation}

\begin{lemma}\label{l:Riemanncurv}
Let $\Riem: \Gamma(TM) \times \Gamma(TM) \times \Gamma(TM) \to \Gamma(TM)$ the Riemannian curvature tensor w.r.t. the Levi-Civita connection. Then
\begin{equation}
\Rcan_{\gamma(t)}v =  \Riem(v,\dot{\gamma}(t))\dot{\gamma}(t), \qquad v \in T_{\gamma(t)}M.
\end{equation}
\end{lemma}

In terms of the above parallel transported frame, the Jacobi equation~\eqref{eq:Jacobicoord} reduces to the classical one for Riemannian structures:
\begin{equation}
\ddot{q}+R(t)q = 0, \qquad \text{with} \qquad R_{ij}(t) = g\left(\Riem(\dot{\gamma}(t),f_i(t))f_j(t),\dot{\gamma}(t)\right).
\end{equation}

\subsection{Canonical frame: sub-Riemannian case}\label{s:canSR}

In the (general) sub-Riemannian setting, such a drastic simplification cannot be achieved, as the structure is more singular and more non-trivial invariants appear. Yet it is possible to simplify as much as possible the Jacobi equation (seen as a first order equation $\dot{X} = 0$) for a sufficiently regular extremal. This is achieved by the so-called \emph{canonical Darboux frame}, introduced in \cite{lizel,lizel2,agrafeedback}.

In particular, $C_1(t)$ and $C_2(t)$ can be again put in a constant, normal form. However, in sharp contrast with the Riemannian case, their very structure may depend on the extremal. Besides, the remaining block $R(t)$ is still non-constant (in general) and \emph{defines} a canonical curvature operator along $\gamma(t)$. As in the Riemannian case, this frame is unique up to constant, orthogonal transformations that preserve the ``normal forms'' of $C_1$ and $C_2$.

At this point, the discussion can be simplified as, for contact sub-Riemannian structures, $C_1$ and $C_2$ have only one possible form that is the same for all non-trivial extremal. We refer to the original paper \cite{lizel} and the more recent \cite{curvature,BR-comparison,BR-connection} for a discussion of the general case and the relation between the normal forms and the invariants of the geodesic (the so-called \emph{geodesic flag}).

\begin{proposition}\label{p:canonical}
Let $\lambda(t)$ be an integral line of a contact sub-Riemannian Hamiltonian. There exists a smooth moving frame along $\lambda(t)$
\begin{equation}
E(t) = (E_0(t),E_1(t),\ldots, E_{2d}(t))^*, \qquad F(t) = (F_0(t),F_1(t),\ldots, F_{2d}(t))^*
\end{equation}
such that the following hold true for any $t$:
\begin{itemize}
\item[(i)] $\spn\{E_{0}(t),\ldots ,E_{2d}(t)\} = \mathcal{V}_{\lambda(t)}$.
\item[(ii)] It is a Darboux basis, namely
\begin{equation}
\sigma(E_\mu,E_\nu) = \sigma(F_\mu,F_\nu)= \sigma(E_\mu,F_\nu) - \delta_{\mu\nu} = 0\, \qquad \mu,\nu =0,\ldots,2d.
\end{equation}
\item[(iii)] The frame satisfies the \emph{structural equations}
\begin{align}
\dot{E}(t) & = C_1 E(t) - C_2 F(t),  \label{eq:struct1}\\
\dot{F}(t) & = R(t) E(t) - C_1^* F(t), \label{eq:struct2}
\end{align}
where $C_1$, $C_2$ are $(2d+1)\times (2d+1)$ matrices defined by
\begin{equation}
C_1 = \begin{pmatrix}
0 & 1 & 0 \\ 0 & 0 & 0 \\ 0 & 0 & 0_{2d-1}
\end{pmatrix}, \qquad C_2 = \begin{pmatrix}
0 & 0 & 0 \\ 0 & 1 & 0 \\ 0 & 0 &\mathbb{I}_{2d-1}
\end{pmatrix},
\end{equation}
and $R(t)$ is a $(2d+1)\times (2d+1)$ smooth family of symmetric matrices of the form
\begin{equation}
R(t) = \begin{pmatrix}
R_{00}(t) &  0 & R_{0j}(t) \\
0 & R_{11}(t) & R_{1j}(t) \\
R_{i0}(t) & R_{i1}(t) & R_{ij}(t)
\end{pmatrix}, \qquad i,j=2,\ldots,2d.
\end{equation}
Notice that $C_1$ is nilpotent, and $C_2$ is idempotent.
\end{itemize}
Moreover, the projections $f_i(t):=\pi_* F_i(t)$ for $i=1,\ldots,2d$ are an orthonormal frame for $\distr_{\gamma(t)}$ and $f_0(t):=\pi_* F_0(t)$ is transverse to $\distr_{\gamma(t)}$.

If $\{\wt{E}(t),\wt{F}(t)\}$ is another frame that satisfies (i)-(iii) for some matrix $\wt{R}(t)$, then there exists a constant orthogonal matrix $O$ that preserves the structural equations (i.e.  $OC_i O^* = C_i$) and
\begin{equation}
\wt{E}(t) = O E(t), \qquad \wt{F}(t) = O F(t), \qquad \wt{R}(t) = O R(t) O^*.
\end{equation}
\end{proposition}

\subsection{Invariant subspaces and curvature}\label{s:invariantspaces}
Let $f_\mu(t)= \pi_* F_\mu(t)$ a frame for $T_{\gamma(t)}M$ (here $\mu = 0,\ldots,2d$). The uniqueness part of Proposition~\ref{p:canonical} implies that this frame is unique up to a constant rotation of the form
\begin{equation}
O = \pm \begin{pmatrix}
1 & 0 & 0 \\
0 & 1 & 0 \\
0 & 0 & U
\end{pmatrix}, \qquad U \in \mathrm{O}(2d-1),
\end{equation}
as one can readily check by imposing the conditions $OC_i O^* = C_i$. In particular, the following invariant subspaces of $T_{\gamma(t)} M$ are well defined
\begin{align}
S_{\gamma(t)}^a & := \spn\{ f_0 \}, \\
S_{\gamma(t)}^b & := \spn\{ f_1 \}, \\
S_{\gamma(t)}^c & := \spn\{ f_2,\ldots,f_{2d}\},
\end{align}
and do not depend on the choice of the canonical frame (but solely on the geodesic $\gamma(t)$). The sets of indices $\{0\}$, $\{1\}$ and $\{2,\ldots,2d\}$ play distinguished roles and, in the following, we relabel these groups of indices as: $a = \{0\}$, $b = \{1\}$ and $c= \{2,\ldots,2d\}$. In particular we have:
\begin{equation}
T_{\gamma(t)}M = S_{\gamma(t)}^{a} \oplus S_{\gamma(t)}^{b} \oplus S_{\gamma(t)}^{c}.
\end{equation}
With this notation, the curvature matrix $R(t)$ has the following block structure
\begin{equation}\label{eq:grouppati}
R(t) = \begin{pmatrix}
R_{aa}(t) &  0 & R_{ac}(t) \\
0 & R_{bb}(t) & R_{bc}(t) \\
R_{ca}(t) & R_{cb}(t) & R_{cc}(t)
\end{pmatrix},
\end{equation}
where $R_{aa}(t)$, $R_{bb}(t)$ are $1\times 1$ matrices, $R_{ac}(t)=R_{ca}(t)^*,R_{bc}(t)= R_{cb}(t)^*$  are $1 \times (2d-1)$ matrices and $R_{cc}$ is a $(2d-1)\times (2d-1)$ matrix.
\begin{definition}
The \emph{canonical curvature} is the symmetric operator $\Rcan_{\gamma(t)}: T_{\gamma(t)} M \to T_{\gamma(t)} M$, that, in terms of the basis $f_a,f_b,f_{c_2},\ldots,f_{c_{2d}}$ is represented by the matrix $R(t)$.
\end{definition}
The definition is well posed, in the sense that different canonical frames give rise to the same operator. For $\alpha,\beta= a,b,c$, we denote by $\mathfrak{R}_{\gamma(t)}^{\alpha\beta} : S^\alpha_{\gamma(t)} \to S^\beta_{\gamma(t)}$ the restrictions of the canonical curvature to the appropriate invariant subspace.
\begin{definition}
The \emph{canonical Ricci curvatures} are the partial traces
\begin{equation}
\Riccan^\alpha_{\gamma(t)}:= \trace \left( \Rcan_{\gamma(t)}^{\alpha\alpha} : S_{\gamma(t)}^\alpha \to S_{\gamma(t)}^\alpha \right), \qquad \alpha = a,b,c.
\end{equation}
\end{definition}

The canonical curvature $\mathfrak{R}(t)$ contains all the information on the germ of the structure along the geodesic. More precisely, consider two pairs $(M,g,\gamma)$ and $(M',g',\gamma')$ where $(M,g)$ and $(M',g')$ are two contact sub-Riemannian structures and $\gamma,\gamma'$ two geodesics. Then $\mathfrak{R}_{\gamma(t)}$ is congruent to $\mathfrak{R}'_{\gamma'(t)}$ if and only if the linearizations of the respective geodesic flows (as flows on the cotangent bundle, along the respective extremals) are equivalent (i.e. symplectomorphic).
\section{Geodesic cost and its asymptotic}

We start by a characterization of smooth points of the squared distance $\dist^2$ on a contact manifold. Let $x_0 \in M$, and let $\Sigma_{x_{0}}\subset M$  be the set of points $x$ such that there exists a unique length minimizer $\gamma:[0,1]\to M$ joining $x_{0}$ with $x$, that is non-conjugate.

\begin{theorem}[see \cite{agrachevsmooth}] \label{t:d2sr}
Let $x_{0}\in M$ and set $\f:= 	\frac{1}{2}\dist^{2}(x_{0},\cdot)$. Then $\Sigma_{x_0}$ is open, dense and $\f$ is smooth on $\Sigma_{x_{0}}$. Moreover if $x\in \Sigma_{x_{0}}$ then $d_{x}\f=\lambda(1)$, where $\lambda(t)$ is the normal extremal of the unique length minimizer $\gamma(t)$ joining $x_{0}$ to $x$ in time 1.
\end{theorem}
\begin{remark}
By homogeneity of the Hamiltonian, if in the statement of the previous theorem one considers a geodesic $\gamma:[0,T]\to M$ joining $x_{0}$ and $x$ in time $T$, then $d_{x}\f= T\lam(T)$.
\end{remark}
\begin{remark}
The statement of Theorem \ref{t:d2sr} is valid on a general sub-Riemannian manifold, with $\Sigma_{x_0}$ defined as the set of points $x$ such that there exists a unique minimizer joining $x_0$ with $x$, which must be not abnormal, and such that $x$ is not conjugate with $x_0$ (see \cite{nostrolibro,curvature,agrachevsmooth,RT-Sard}).
\end{remark}

\begin{remark} \label{cutzero} 
If $M$ is a contact sub-Riemannian manifold, the function $\f=\frac{1}{2}\dist^{2}(x_{0},\cdot)$ is Lipschitz on $M\setminus\{x_{0}\}$, due to the absence of abnormal minimizers (see for instance \cite{nostrolibro,noterifford}). In particular from this one can deduce that the set $M\setminus \Sigma_{x_{0}}$ where $\f$ is not smooth has measure zero.
\end{remark}

\begin{definition}
Let $x_{0}\in M$ and consider a geodesic $\gamma:[0,T]\to M$ such that $\gamma(0)=x_{0}$. The \emph{geodesic cost} associated with $\gamma$ is the family of functions for $t>0$
\begin{equation}
\cc_{t}(x):=-\frac{1}{2t}\dist^{2}(x,\gamma(t)),\qquad x\in M.
\end{equation}
\end{definition}

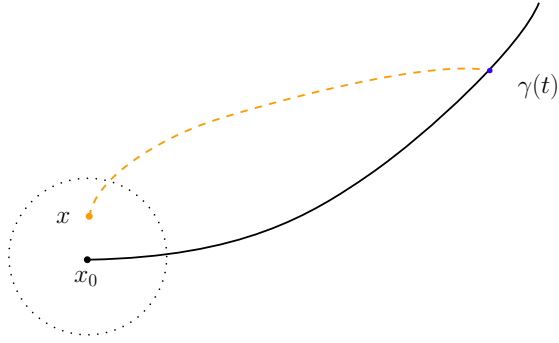
\begin{figure}[t]
\centering
\scalebox{0.6} 
{
\begin{pspicture}(0,-3.7)(12.561894,3.68)
\definecolor{color193}{rgb}{1.0,0.6,0.0}
\definecolor{color192}{rgb}{0.2,0.0,1.0}
\psbezier[linewidth=0.04,linecolor=color193,linestyle=dashed,dash=0.16cm 0.16cm,dotsize=0.07055555cm 2.0]{-*}(10.54,2.16)(9.14,2.44)(5.3795967,1.3622835)(4.42,1.04)(3.4604032,0.7177165)(1.96,-0.1)(1.74,-1.06)
\psbezier[linewidth=0.04,dotsize=0.07055555cm 2.0]{-*}(11.6,3.66)(11.219311,2.6917758)(9.175264,0.7630423)(7.88,-0.14)(6.5847354,-1.0430423)(5.088138,-1.9508411)(1.7,-2.02)
\fontsize{16}{0}
\usefont{T1}{ptm}{m}{n}
\rput(1.671455,-2.435){$x_0$}
\usefont{T1}{ptm}{m}{n}
\rput(11.591455,1.805){$\gamma(t)$}
\pscircle[linewidth=0.04,linestyle=dotted,dotsep=0.16cm,dimen=inner](1.71,-1.95){1.71}
\psdots[dotsize=0.12,linecolor=color192](10.52,2.16)
\usefont{T1}{ptm}{m}{n}
\rput(1.171455,-1.095){$x$}
\end{pspicture} 
}
\caption{The geodesic cost function.}
\end{figure}

From Theorem \ref{t:d2sr} one obtains smoothness properties of the geodesic cost.
\begin{theorem} \label{t:mst} 
Let $x_{0}\in M$ and $\gamma(t)=\Exp_{x_{0}}(t\lam_{0})$ be a geodesic. Then there exists $\eps>0$ and an open set $U\subset (0,\eps)\times M$ such that 
\begin{itemize}
\item[(i)] $(t,x_{0})\in U$ for all $t\in (0,\eps)$,
\item[(ii)] the function $(t,x)\mapsto \cc_{t}(x)$ is smooth on $U$,
\item[(iii)] for any $(t,x) \in U$, the covector $\lambda_x = d_x c_t$ is the initial covector of the unique geodesic connecting $x$ with $\gamma(t)$ in time $t$.
\end{itemize}
In particular $\lambda_0 = d_{x_0} c_t$ and $x_{0}$ is a critical point for the function $\dot{\cc}_{t}:=\frac{d}{dt}c_{t}$ for every $t\in(0,\eps)$.
\end{theorem}

\begin{proof}
There exists $\eps>0$ small enough such that for $t\in (0,\eps)$, the curve $\gamma|_{[0,t]}$ is the unique minimizer joining $x_{0}=\gamma(0)$ and $\gamma(t)$, and $\gamma(t)$ is non conjugate to $\gamma(0)$.
As a direct consequence of Theorem \ref{t:d2sr} one gets $\cc_{t}(x)$ is smooth for fixed $t\in (0,\eps)$ and $x$ in a neighborhood of $x_{0}$. The fact that the function $c_{t}(x)$ is smooth on an open set $U$ as a function of the two variables is proved in \cite[Appendix A]{curvature}.
Let us prove (iii). Notice that 
\begin{equation} 
d_{x_{0}}c_{t}=-\frac{1}{t}d_{x_{0}}\left(\frac{1}{2}\dist^{2}(\gamma(t),\cdot)\right)=-\frac{1}{t}d_{x_{0}}\f_{t}. 
\end{equation}
where $\f_{t}$ denotes one half of the squared distance from the point $\gamma(t)$. Observe that $x\in \Sigma_{\gamma(t)}$ is in the set of smooth points of the squared distance from $\gamma(t)$. Hence the differential $d_{x}\f_{t}$ is the final covector of the unique geodesic joining $\gamma(t)$ with $x$ in time $1$. Thus, $-d_x \f_t$ is the initial covector of the unique geodesic joining $x$ with $\gamma(t)$ in time $1$, and  $d_x c_t = -\frac{1}{t} d_x \f_t$ is the initial covector $\lambda_0$ of the unique geodesic connecting $x$ with $\gamma(t)$ in time $t$.
\end{proof}

\subsection{A geometrical interpretation}\label{s:interpr}
By Theorem \ref{t:mst}, for each $(t,x) \in U$ the function $x\mapsto \dot c_{t}(x)$ has a critical point at $x_{0}$. This function will play a crucial role in the following, and has a nice geometrical interpretation. Let $W_{x,\gamma(t)}^t \in T_{\gamma(t)}^*M$ be the final tangent vector of the unique minimizer connecting $x$ with $\gamma(t)$ in time $t$. We have:
\begin{equation}\label{eq:interpr}
\dot{c}_t(x) = \frac{1}{2}\|W_{x_0,\gamma(t)}^t - W_{x,\gamma(t)}^t\|^2-\frac{1}{2}\|W_{x_0,\gamma(t)}^t\|^2.
\end{equation}
This formula is a consequence of Theorem~\ref{t:d2sr} and is proved in \cite[Appendix I]{curvature}. The second term $\|W_{x_0,\gamma(t)}^t\|$ is the speed of the geodesic $\gamma$ and is then an inessential constant. Eq.~\eqref{eq:interpr} has also natural physical interpretation as follows. Suppose that $A$ and $B$ live at points $x_A$ and $x_B$ respectively (see Fig.~\ref{fig:interpr}). Then $A$ chooses a geodesic $\gamma(t)$, starting from $x_A$, and tells $B$ to meet at some point $\gamma(t)$ (at time $t$). Then $B$ must choose carefully its geodesic in order to meet $A$ at the point $\gamma(t)$ starting from $x_B$, following the curve for time $t$. When they meet at $\gamma(t)$ at time $t$, they compare their velocity by computing the squared norm (or \emph{energy}) of the difference of the tangent vectors. This gives the value of the function $\dot{c}_t$, up to a constant.

The ``curvature at $x_0$'' is encoded in the behavior of this function for small $t$ and $x$ close to $x_{0}$. To support this statement, consider the Riemannian setting, in which Eq.~\eqref{eq:interpr} clearly remains true. If the Riemannian manifold is positively (resp. negatively) curved, then the two tangent vectors, compared at $\gamma(t)$, are more (resp. less) divergent w.r.t. the flat case (see Fig.~\ref{fig:interpr}).
\begin{remark}
We do not need parallel transport: $A$ and $B$  meet at $\gamma(t)$ and make \emph{there} their comparison. We only used the concept of ``optimal trajectory'' and ``difference of the cost''. This interpretation indeed works for a general optimal control system, as in the general setting of \cite{curvature}.
\end{remark}
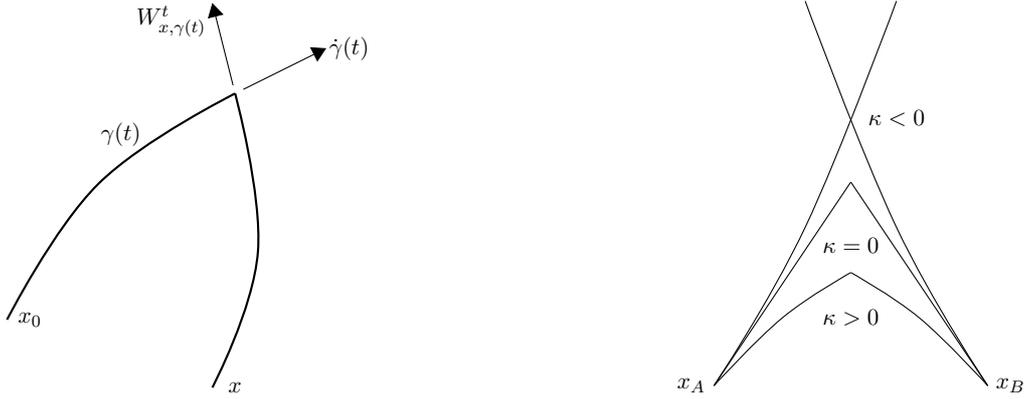
\begin{figure}
\centering
\begin{tikzpicture}[scale=0.6,every node/.style={scale=0.8}]
\draw[thick]  plot[smooth, tension=.7] coordinates {(-2,-2.5) (0,0.5) (3,2.5)};
\draw[thick]  plot[smooth, tension=.7] coordinates {(3,2.5) (3.5,-1) (2.5,-4)};
\draw (3,2.5) node (v1) {};
\draw [-triangle 60](v1) -- (2.5,4.5);
\draw [-triangle 60](v1) -- (5,3.5);
\node at (-1.5,-2.5) {$x_0$};
\node at (3,-4) {$x$};
\node at (0.5,1.6) {$\gamma(t)$};
\node at (5.5,3.5) {$\dot\gamma(t)$};
\node at (1.6,4.1) {$W_{x,\gamma(t)}^{t}$};
\end{tikzpicture}\qquad\qquad\qquad\qquad\qquad
\begin{tikzpicture}[scale=0.6,every node/.style={scale=0.8}]
\draw (11,0.5) -- (8,-4) node (v2) {};
\draw (11,0.5) -- (14,-4) node (v3) {};
\draw  plot[smooth, tension=.7] coordinates {(v2)};
\draw  plot[smooth, tension=.7] coordinates {(v2) (9.5,-2.5) (11,-1.5)};
\draw  plot[smooth, tension=.7] coordinates {(v3) (12.5,-2.5) (11,-1.5)};
\draw  plot[smooth, tension=.7] coordinates {(8,-4) (10,-0.5) (12,4.5)};
\draw  plot[smooth, tension=.7] coordinates {(v3) (12,-0.5) (10,4.5) };
\node at (7.5,-4) {$x_A$};
\node at (14.5,-4) {$x_B$};
\node at (11,-2.5) {$\kappa>0$};
\node at (11,-0.9) {$\kappa=0$};
\node at (12.,1.9) {$\kappa<0$};
\end{tikzpicture}
\caption{Geometric interpretation of $\dot{c}_t$. $W_{x,\gamma(t)}^t$ is the tangent vector, at time $t$, of the unique geodesic connecting $x$ with $\gamma(t)$ in time $t$.}\label{fig:interpr}
\end{figure}

\subsection{A family of operators}
Let $f:M\to \R$ be a smooth function. Its first differential at a point $x\in M$ is the linear map $d_{x}f:T_{x}M\to \R$. The \emph{second differential} of $f$ is well defined only at a critical point, i.e. at those points $x$ such that $d_{x}f=0$. In this case the map
\begin{equation}
d^{2}_{x}f: T_{x}M\times T_{x}M\to \R,\qquad d^{2}_{x}f(v,w)=V(Wf)(x),
\end{equation}
where $V,W$ are vector fields such that $V(x)=v$ and $W(x)=w$, respectively, is a well defined symmetric bilinear form that does not depend on the choice of the extensions. The associated quadratic form, that we denote by the same symbol $d^2_x f: T_x M \to \mathbb{R}$, is defined by
\begin{equation}
d^{2}_{x}f(v)=\frac{d^{2}}{dt^{2}}\bigg|_{t=0} f(\gamma(t)),\qquad \text{where} \quad \gamma(0)=x,\quad \dot \gamma(0)=v.
\end{equation}
By Theorem \ref{t:mst}, for small $t>0$ the function $x\mapsto \dot c_{t}(x)$ has a critical point at $x_{0}$. Hence we can consider the family of quadratic forms restricted on the distribution
\begin{equation} 
d^{2}_{x_0}\dot c_{t}\big|_{\distr_{x_0}}: \distr_{x_0} \to \R,\qquad \all t>0. 
\end{equation}
Using the sub-Riemannian scalar product on $\distr_{x_0}$ we can associate with this family of quadratic forms the family of symmetric operators 
$\QQ_\lam(t):\distr_{x_0} \to \distr_{x_0}$ defined for small $t>0$ by
\begin{equation}\label{eq:prscl}
d^{2}_{x_0}\dot c_{t}(w)=\metr{ \QQ_\lam(t)w}{w},\qquad  \all w\in \distr_{x_0},
\end{equation}
where $\lambda = d_{x_0} c_t$ is the initial covector of the fixed geodesic $\gamma$.

\begin{theorem}\label{t:main} 
Let $M$ be a contact sub-Riemannian manifold and $\gamma(t)$ be a non-trivial geodesic with initial covector $\lam\in T^{*}_{x_0}M$. Let $\QQ_\lam(t):\distr_{x_0} \to \distr_{x_0}$ be as in \eqref{eq:prscl}.

The family of operators $t\mapsto t^{2}\QQ_\lam(t)$  can be extended to a smooth family of operators on $\distr_{x_0}$ for small $t\geq 0$, symmetric with respect to $g$. Moreover,
\begin{equation}
\Qz_{\lam}:=\displaystyle \lim_{t\to 0^+}t^{2}\QQ_\lam(t) \geq \id > 0.
\end{equation}
Then we have the following Laurent expansion at $t=0$:
\begin{equation}\label{eq:mainexp}
\QQ_\lam(t)= \frac{1}{t^{2}}\Qz_{\lam}+ \sum_{i=0}^{m}\QQ^{(i)}_\lam t^{i}+O(t^{m+1}).
\end{equation}
\end{theorem}
In particular, Theorem~\ref{t:main} \emph{defines} a sequence of operators $\QQ^{(i)}_\lambda$, for $i \in \N$. These operators, together with $\Qz_\lam$, contain all the information on the germ of the structure along the fixed geodesic (clearly any modification of the structure that coincides with the given one on a neighborhood of the geodesic gives the same family of operators).

\begin{remark} 
Applying this construction to a Riemannian manifold, one finds that $\Qz_\lam = \mathbb{I}$ for any geodesic, and the operator $\QQ^{(0)}_\lam$ is indeed the ``directional'' sectional curvature in the direction of the geodesic (see \cite[Sect. 4.4.2]{curvature}):
\begin{equation}
\QQ^{(0)}_\lam=\frac13 \Riem(\cdot,\dot\gamma)\dot\gamma,
\end{equation}
where $\dot\gamma$ is the initial velocity vector of the geodesic (that is the vector corresponding to $\lam$ in the canonical isomorphism $T_{x}^{*}M\simeq T_{x}M$ defined by the Riemannian metric).
\end{remark}

The expansion is interesting in two directions. First, the singularity is controlled ($\QQ_\lambda(t)$ has a second order pole at $t=0$), even though the sub-Riemannian squared distance is not regular on the diagonal. Second, the Laurent polynomial of second order of $\QQ_\lambda(t)$ recovers the whole canonical curvature operator $\mathfrak{R}$ at the point. The next two theorems describe in detail the Laurent expansion of the operator $\QQ_\lam(t)$. 

\begin{theorem} \label{t:main2} The symmetric operator $\Qz_{\lam}: \distr_{x_{0}}\to \distr_{x_{0}}$ satisfies
\begin{itemize}
\item[(i)] $\spec\Qz_{\lam}=\{1,4\}$,
\item[(ii)] $\trace\Qz_{\lam}=2d+3$.  
\end{itemize}
The vector subspaces of $\distr_{\gamma(0)}$ 
\begin{equation}
S^b_{\gamma(0)}:=\spn\{J \dot\gamma(0)\}, \qquad S^c_{\gamma(0)}:=J\dot\gamma(0)^\perp \cap \distr_{\gamma(0)}
\end{equation}
are the eigenspaces of $\Qz_{\lam}$ corresponding to eigenvalues $4$ and $1$, respectively. In particular the eigenvalue $4$ has multiplicity $1$ while the eigenvalue $1$ has multiplicity $2d-1$.
\end{theorem}
 
\begin{remark}
Let $K_{\gamma(0)} \subset \distr_{\gamma(0)}$ be the hyperplane $d\omega$-orthogonal to $\dot\gamma(0)$, that is 
\begin{equation}
K_{\gamma(0)}:=\{v \in \distr_{\gamma(0)}\mid d\omega(v,\dot\gamma(0)) = 0\}.
\end{equation}
Then $S^c_{\gamma(0)} =K_{\gamma(0)}$ and $S^b_{\gamma(0)} = K_{\gamma(0)}^\perp\cap \distr_{\gamma(0)}$.
\end{remark}
\begin{theorem} \label{t:main3}
The asymptotics of $\QQ_{\lam}(t)$ recovers the canonical curvature at $\gamma(0)$. More precisely, according to the orthogonal decomposition $\distr_{\gamma(0)} = S^b_{\gamma(0)} \oplus S^c_{\gamma(0)}$ we have
\begin{equation}
\mathcal{Q}_\lambda^{(0)} = \begin{pmatrix}
 \frac{2}{15} R_{bb} & \frac{1}{12} R_{bc} \\[0.1cm]
 \frac{1}{12} R_{cb} & \frac{1}{3} R_{cc} \\  
\end{pmatrix},
\end{equation}
\begin{equation}
\mathcal{Q}_\lambda^{(1)} = \begin{pmatrix}
 \frac{1}{15} \dot R_{bb} & \frac{1}{10} R_{ac}-\frac{1}{30} \dot R_{bc} \\[0.1cm]
 \frac{1}{10} R_{ca}-\frac{1}{30} \dot R_{cb} & \frac{1}{6} \dot R_{cc} \\
\end{pmatrix},
\end{equation}
\begin{equation}
\mathcal{Q}_\lambda^{(2)} = \frac{1}{240}\begin{pmatrix}
 \frac{1}{35} \left(240 R_{aa}+ 44 R_{bb}^2+65 R_{bc} R_{cb}+240 \ddot R_{bb} \right) & R_{bb} R_{bc}-2  R_{bc} R_{cc}+12 \dot R_{ac}-6 \ddot R_{bc} \\
 R_{cb}  R_{bb} -2 R_{cc} R_{cb}+12 \dot R_{ca}-6 \ddot R_{cb} & 16 R_{cc}^2+R_{bc} R_{cb}+12 \ddot R_{cc}
\end{pmatrix},
\end{equation}
where the dot denotes the derivative w.r.t. $t$ and here $R_{\alpha\beta} = R_{\alpha\beta}(0)$ for $\alpha,\beta =a,b,c$ are the components of the canonical curvature operator $\Rcan_{\gamma(0)}:T_{x_0} M \to T_{x_0} M$.
\end{theorem}
\begin{remark}
To recover the operator $\mathfrak{R}_{\gamma(s)} : T_{\gamma(s)} M \to T_{\gamma(s)} M$ for all $s$ on the geodesic, one needs the Laurent expansion of the operator $\QQ_{\lambda(s)}(t)$ for all the points $\lambda(s)$ on the extremal.
\end{remark}
\begin{proof}[Proof of Theorems.~\ref{t:main}, \ref{t:main2} and \ref{t:main3}]
For a $m$-tuple $v = (v_1,\ldots,v_m)$ of vectors and a linear map $L:V \to W$, the notation $L(v)$ denotes the $m$-tuple $(Lv_1,\ldots,Lv_m)$. Consider a canonical frame $E(t),F(t)$ from Proposition~\ref{p:canonical}. Accordingly, $f(t):= \pi_* F(t) \in T_{\gamma(t)} M$ is the $(2d+1)$-tuple
\begin{equation}
f(t) = (f_0(t),\underbrace{f_1(t),\ldots,f_{2d}(t)}_{\text{o.n. basis for $\distr_{\gamma(t)}$}})^*.
\end{equation}
The geodesic cost function $c_t$ is regular at $x_0$, and thus its second differential is not well defined as a quadratic form on $T_{x_0} M$. Still it is well defined as a linear map as follows. Consider the differential $d c_t : M \to T^*M$. Taking again the differential at $x_0$, we get a map $d^2_{x_0} c_t : T_{x_0} M \to T_{\lambda}(T^*M)$, where $\lambda = d_{x_0} c_t$ is the initial covector associated with the geodesic.

For fixed $t>0$ and $x$ in a neighborhood $U_t$ of $x_0$, we clearly have $\pi ( d_x c_t) = x$. Thus $d^2_{x_0} c_t$ has maximal rank. Similarly, $\pi(e^{t\vec{H}} d_x c_t) = \gamma(t)$ by Theorem~\ref{t:mst}. It follows that $e^{t\vec{H}}_*  d^2_{x_0}c_t (T_{x_0}M)  = \mathcal{V}_{\lambda(t)}$. Then there exists a smooth family of $(2d+1)\times (2d+1)$ matrices $\Theta(t)$ such that
\begin{equation}\label{eq:verticality}
e^{t\vec{H}}_*  d^2_{x_0} c_t(f(0)) = \Theta(t) E(t).
\end{equation}
We pull back the canonical frame $E(t)$ through the Hamiltonian flow, and express this in terms of the canonical frame at time zero:
\begin{equation}\label{eq:expansion}
e^{-t\vec{H}}_* E(t) = A(t)E(0) + B(t) F(0), \qquad t \geq 0.
\end{equation}
for some smooth families of $(2d+1)\times (2d+1)$ matrices $A(t),B(t)$ (in particular $A(0) = \mathbb{I}$ and $B(0) = 0$). As we already noticed, $\pi (d_x c_t )= x$ (here $t>0$), then
\begin{align}
f(0) & = \pi_* d^2_{x_0} c_t(f(0))  & \\
& =\Theta(t) \pi_* e^{-t{\vec{H}}}_* E(t) & (\text{by Eq.~\eqref{eq:verticality}}) \\
& = \Theta(t) \pi_* (A(t)E(0) + B(t) F(0)) & (\text{by Eq.~\eqref{eq:expansion}}) \\
& = \Theta(t) B(t) f(0) & \text{(by definition of $f$ and verticality of $E$})
\end{align}
Thus $\Theta(t) B(t) = \mathbb{I}$ (in particular, $B(t)$ is not-degenerate for small $t>0$). Then
\begin{equation}
d^2_{x_0} c_t(f(0)) = B(t)^{-1} A(t) E(0) + F(0), \qquad t >0.
\end{equation}
Let $S(t) := A(t)^{-1} B(t)$. We get, taking a derivative
\begin{equation}
d^2_{x_0} \dot{c}_t(f(0)) = \frac{d}{dt}\left[S(t)^{-1}\right] E(0).
\end{equation}
In particular, $d^2_{x_0} \dot{c}_t$ maps $T_{x_0}M$ to the vertical subspace $\mathcal{V}_\lambda = T_{\lambda} (T_{x_0}^*M) \simeq T_{x_0}^*M$. By the latter identification, we recover the standard second differential (the Hessian) $d^2_{x_0} \dot{c}_t : T_{x_0}M \to T_{x_0}^*M$ at the critical point $x_0$. The associated quadratic form is, for $t>0$,
\begin{equation}\label{eq:notrestricted}
\langle d^2_{x_0} \dot{c}_t(f(0)), f(0)\rangle = \frac{d}{dt}\left[S(t)^{-1}\right].
\end{equation}
According to \eqref{eq:prscl}, $\QQ_\lambda(t) : \distr_{x_0} \to \distr_{x_0}$ is the operator associated with the quadratic form $d^2_{x_0} \dot{c}_t : T_{x_0} M \to T_{x_0}^*M$, restricted on $\distr_{x_0}$, through the sub-Riemannian product. From Proposition~\ref{p:canonical}, that $f = (f_0,f_1,\ldots,f_{2d})$ where $f_1,\ldots,f_{2d}$ is an orthonormal frame for $\distr$. Then, in terms of the basis $f_1,\ldots,f_{2d}$ of $\distr_{x_0}$ (here we suppressed explicit evaluation at $t=0$) we have, from~\eqref{eq:notrestricted}:
\begin{equation}
\QQ_\lambda(t) = \frac{d}{dt}\left[S(t)^{-1}\right]_\square.
\end{equation}
where the notation $M_\square$ denotes the bottom right $2d\times 2d$ block of a matrix $M$.

We are left to compute the asymptotics of the matrix $S(t)^{-1} =B(t)^{-1}A(t)$, where $A(t)$ and $B(t)$ are defined by Eq.~\eqref{eq:expansion}. In order to do that we study a more general problem. We introduce smooth matrices $C(t)$, $D(t)$ such that (here $t\geq 0$)
\begin{align}
e^{-t\vec{H}}_* E(t)  &= A(t)E(0) + B(t) F(0), \\
e^{-t\vec{H}}_* F(t)  &= C(t)E(0) + D(t) F(0).
\end{align}
Since $(E(t),F(t))$ satisfies the structural equations~\eqref{eq:struct1}-\eqref{eq:struct2}, we find that $A,B,C,D$ are the solutions of the following Cauchy problem:
\begin{align}\label{eq:Cauchyp}
\frac{d}{dt} \begin{pmatrix}
A & B \\
C & D
\end{pmatrix} = \begin{pmatrix}
C_1 &  -C_2 \\ R(t) & -C_1^*
\end{pmatrix} \begin{pmatrix}
A & B \\
C & D
\end{pmatrix}, \qquad \begin{pmatrix}
A(0) & B(0) \\
C(0) & D(0)
\end{pmatrix} = \begin{pmatrix}
\mathbb{I} & 0 \\
0 & \mathbb{I}
\end{pmatrix}.
\end{align}
From~\eqref{eq:Cauchyp} one can compute an arbitrarily high order Taylor polynomial of $A(t)$ and $B(t)$, needed for the computation of $S(t)^{-1}$. 

Notice that $A(0) = \mathbb{I}$, but $B(0) = 0$ (but $B(t)$ is non-singular as soon as $t > 0$). We say that $B(t)$ has order $r$ (at $t=0$) if $r$ is the smallest integer such that the $r$-th order Taylor polynomial of $B(t)$ (at $t=0$) is non-singular. It follows from explicit computation that $B(t)$ has order $3$:
\begin{equation}
B(t) = -t
\begin{pmatrix}
0 & 0 & 0 \\ 0 & 1 & 0 \\ 0 & 0 & \mathbb{I}_{2d-1}
\end{pmatrix} + 
\frac{t^2}{2}\begin{pmatrix}
0 & -1 & 0 \\ 1 & 0 & 0 \\ 0 & 0 & 0
\end{pmatrix} + 
\frac{t^3}{6}  \begin{pmatrix}
1 & 0 & 0 \\ 0 & 0 & 0 \\ 0 & 0 & 0
\end{pmatrix} + 
O(t^4).
\end{equation}
In particular, for $t>0$, the Laurent polynomial of $B(t)^{-1}$ has poles of order $\geq 3$. Thus, to compute the Laurent polynomial of $S(t)^{-1} = B(t)^{-1}A(t)$ at order $N$, one needs to compute $B(t)$ and $A(t)$ at order $N+6$ and $N+3$, respectively. After explicit and long computations, one finds
\begin{equation}
S(t)^{-1} = \frac{1}{t^3} \begin{pmatrix}
12 & 0 & 0 \\ 0 & 0 & 0 \\ 0 & 0 & 0
\end{pmatrix} + 
\frac{1}{t^2}\begin{pmatrix}
0 & -6 & 0 \\ -6 & 0 & 0 \\ 0 & 0 & 0
\end{pmatrix} + 
\frac{1}{t}
\begin{pmatrix}
\tfrac{6}{5}R_{bb} & 0 & 0 \\ 0 & -4 & 0 \\ 0 & 0 & -\mathbb{I}_{2d-1}
\end{pmatrix} + O(1).
\end{equation}
In particular, for the restriction we obtain
\begin{equation}
\left[S(t)^{-1}\right]_\square = -\frac{1}{t}
\begin{pmatrix}
4 & 0 \\  0 & \mathbb{I}_{2d-1}
\end{pmatrix} + O(1).
\end{equation}
The restriction to the distribution has the effect of taking the ``least degenerate'' block. Taking the derivative, we obtain
\begin{equation}\label{eq:principal}
\frac{d}{dt}\left[S(t)^{-1}\right]_\square = \frac{1}{t^2}
\begin{pmatrix}
4 & 0 \\ 0 & \mathbb{I}_{2d-1}
\end{pmatrix} + O(t).
\end{equation}
This proves Theorems~\ref{t:main} and~\ref{t:main2}, as $\Qz_\lambda = \lim_{t\to 0^+} t^2\QQ_{\lambda}(t)$ is finite and $\geq \mathbb{I}$.

For Theorem~\ref{t:main3}, we need the Laurent polynomial of $\QQ_\lambda(t) = \frac{d}{dt}\left[S(t)^{-1}\right]_\square$ at order $2$. That is, we need $S(t)^{-1}$ at order $N=3$. According to the discussion above, this requires the solution of Cauchy problem~\eqref{eq:Cauchyp} at order $N+6 = 9$. This is achieved by long computations (made easier by the properties of $C_1,C_2$: $C_1^2 =0$, $C_2^2 = C_2$, $C_1C_2 = C_1$ and $C_2 C_1 = 0$). One checks that 
\begin{equation}
\frac{d}{dt}\left[S(t)^{-1}\right]_\square = \frac{1}{t^2} \Qz_\lambda + \QQ_\lambda^{(0)} + t\QQ_\lambda^{(1)}+t^2 \QQ_\lambda^{(2)} + O(t^3),
\end{equation}
where $\Qz_\lam$ was computed in Eq.~\eqref{eq:principal} while
\begin{equation}
\mathcal{Q}_\lambda^{(0)} = \begin{pmatrix}
 \frac{2}{15} R_{bb} & \frac{1}{12} R_{bc} \\[0.1cm]
 \frac{1}{12} R_{cb} & \frac{1}{3} R_{cc} \\
\end{pmatrix},
\end{equation}
\begin{equation}
\mathcal{Q}_\lambda^{(1)} = \begin{pmatrix}
 \frac{1}{15} \dot R_{bb} & \frac{1}{10} R_{ac}-\frac{1}{30} \dot R_{bc} \\[0.1cm]
 \frac{1}{10} R_{ca}-\frac{1}{30} \dot R_{cb} & \frac{1}{6} \dot R_{cc} \\
\end{pmatrix},
\end{equation}
\begin{equation}
\mathcal{Q}_\lambda^{(2)} = \frac{1}{240}\begin{pmatrix}
 \frac{1}{35} \left(240 R_{aa}+44 R_{bb}^2+65 R_{bc} R_{cb}+240 \ddot R_{bb} \right) & R_{bb} R_{bc}-2  R_{bc} R_{cc}+12 \dot R_{ac}-6 \ddot R_{bc} \\
 R_{cb}  R_{bb} -2 R_{cc} R_{cb}+12 \dot R_{ca}-6 \ddot R_{cb} & 16 R_{cc}^2+R_{bc} R_{cb}+12 \ddot R_{cc}
\end{pmatrix}.
\end{equation}
where we labeled the indices of $R(t)$ as in~\eqref{eq:grouppati}, that is $a= \{0\}$, $b=\{1\}$ and $c=\{2,\ldots,2d\}$.

This computes the asymptotics of $\QQ_\lambda(t)$ in terms of the orthonormal frame $f_1(0),\ldots,f_{2d}(0)$ obtained as a projection of the canonical frame. By definition (see Sec.~\ref{s:invariantspaces}) $S^b_{\gamma(0)} = \spn\{f_1(0)\}$ and $S^c_{\gamma(0)} = \spn\{f_2(0),\ldots,f_{2d}(0)\}$ and thus concludes the proofs of statements (i)-(ii) of Theorem~\ref{t:main2} and Theorem~\ref{t:main3}. The explicit form of $S^b_{\gamma(0)}$ and $S^c_{\gamma(0)}$ follows from computation of the canonical frame in Theorem~\ref{t:cansplitting}.
\end{proof}

\subsection{Reparametrization and homogeneity} 
The operators $\Qz_{\lam}$ and $\QQ^{(i)}_{\lam}$ for $i\geq0$, are well defined for each nontrivial geodesic on the contact manifold $M$, i.e. for each covector $\lam\in T_{x}^{*}M$ such that $H(\lambda)\neq 0$.  

By homogeneity of the Hamiltonian, if $H(\lambda) \neq 0$, then also $H(\alpha\lambda) = \alpha^2 H(\lambda) \neq 0$ for $\alpha \neq 0$.

\begin{proposition}
For all $\lambda$ such that $H(\lambda) \neq 0$, the operators $\Qz_{\lam},\QQ^{(i)}_{\lam} : \distr_x \to \distr_{x}$ satisfy the following homogeneity properties:
\begin{equation}
\Qz_{\al\lam}=\Qz_{\lam},\qquad \QQ^{(i)}_{\al\lam}=\alpha^{2+i} \QQ^{(i)}_{\lam}. \qquad i \in \N,\, \alpha > 0.
\end{equation}
\end{proposition}

\begin{proof} Here we denote by $x\mapsto c_{t}^{\lam}(x)$ the geodesic cost function associated with $\lam\in T^{*}_{x}M$. For $\al > 0$, if $\Exp_x(t\lam)=\gamma(t)$ then $\Exp_x(t\al \lam)=\Exp_x(\al t \lam)=\gamma(\al t)$. Hence
\begin{equation}
\cc^{\alpha\lam}_{t}(x)=-\frac{1}{2t}\dist^{2}(x,\gamma(\al t))=-\frac{\al}{2\al t}\dist^{2}(x,\gamma(\al t))=\al  \cc^{\lam}_{\al t}(x).
\end{equation}
This implies that $\dot{c}_t^{\alpha\lambda} = \alpha^2 \dot{c}_{\alpha t}^\lambda$ and $d^{2}_{x}\dot{c}_t^{\alpha\lambda} = \alpha^2 d^{2}_{x}\dot{c}_{\alpha t}^\lambda$. If we restrict these quadratic forms to the distribution $\distr_{x}$, we get the identity
\begin{equation} \label{eq:qhomog}
\QQ_{\al\lam}(t)=\al^{2}\QQ_{\lam}(\al t).
\end{equation} 
Applying the expansion of Theorem \ref{t:main} to \eqref{eq:qhomog} we get
\begin{equation} 
\frac{1}{ t^2}\Qz_{\al\lam}+ \sum_{i=0}^{m}\QQ^{(i)}_{\al\lam} t^{i}+O(t^{m+1})=\al^{2}\left(\frac{1}{\al^{2} t^2}\Qz_{\lam} + \sum_{i=0}^{m}\al^{i}\QQ^{(i)}_\lam t^{i}+O(t^{m+1})\right),
\end{equation}
which gives the desired identities.
\end{proof}

\section{Bonnet-Myers type theorems}\label{s:bm}

In \cite{BR-comparison}, under suitable bounds on the canonical curvature of a given sub-Riemannian geodesic, we obtained bounds on the first conjugate time. Then, with global conditions on $\mathfrak{R}$, one obtains Bonnet-Myers type results. We first specify the results of \cite[Sec. 6]{BR-comparison} to contact structures (in the contact case all the necessary hypotheses apply). Then, using the results of Sec.~\ref{s:computations}, we express them in terms of known tensors (Tanno's tensor, curvature and torsion). 

\begin{theorem}\label{t:bonnetmyers2}
Assume that for every unit speed geodesic $\gamma(t)$ of a complete contact sub-Rie\-mannian structure of dimension $2d+1$, with $d >1$, we have
\begin{equation}\label{eq:bound}
\mathfrak{Ric}^{c}_{\gamma(t)}\geq (2d-2)\kappa_{c}, \qquad \forall t \geq 0,
\end{equation}
for some $\kappa_{c}>0$. Then 
the manifold is compact, with sub-Riemannian diameter not greater than $\pi/\sqrt{\kappa_{c}}$, and the fundamental group is finite.
\end{theorem}

\begin{theorem}\label{t:bm-general}
Assume that for every unit speed geodesic $\gamma(t)$ of a complete contact sub-Rie\-mannian structure of dimension $2d+1$, with $d \geq 1$, we have
\begin{equation}
\mathfrak{Ric}_{\g(t)}^{a} \geq \kappa_{a},\quad  \mathfrak{Ric}_{\g(t)}^{b} \geq \kappa_{b}, \qquad \forall t \geq 0,
\end{equation}
for some $\kappa_a,\kappa_b \in \R$ such that
\begin{equation}\label{eq:3dcond}
\begin{cases}
\kappa_{b} > 0, \\ 4\kappa_{a} > -\kappa_{b}^2,
\end{cases}
\qquad \text{or} \qquad
\begin{cases}
\kappa_{b} \leq 0, \\ \kappa_{a} >0.
\end{cases}
\end{equation}
Then the manifold is compact, with sub-Riemannian diameter not greater than $t_*(\kappa_a,\kappa_b)<+\infty$, and the fundamental group is finite.
\end{theorem}
\begin{remark}\label{rmk:tstar}
The value of $t_*(\kappa_a,\kappa_b)$ in Theorem~\ref{t:bm-general} is the first conjugate time of a particular variational problem (a so-called \emph{LQ optimal control problem}) which, in the framework of \cite{BR-comparison}, plays the role of a constant curvature model. These are \emph{synthetic models}, that are not sub-Riemannian manifolds but on which the concept of conjugate time $t_*$ is well defined. It can be explicitly computed as the smallest positive root of a transcendental equation. In this specific case, $t_*(\kappa_a,\kappa_b)$ is the first positive solution of:
\begin{equation}\label{eq:positiveroot}
\frac{4\kappa_a[1-\cos(\lambda_+ t)\cosh(\lambda_- t)] - 2\kappa_b \lambda_+ \lambda_- \sin(t\lambda_+) \sinh(t\lambda_-)}{2\kappa_a(4\kappa_a + \kappa_b^2)}=0, \quad \lambda_{\pm} = \sqrt{\tfrac{\pm \kappa_b + \sqrt{4\kappa_a +\kappa_b^2}}{2}}.
\end{equation}
This equation admits positive roots (actually, an infinite number of them) precisely if and only if \eqref{eq:3dcond} are satisfied. This result, that already in this simple case is quite cumbersome to check, comes from an general theorem on the existence of conjugate times for LQ optimal control problems, obtained in \cite{ARS-LQ}. For $\kappa_a =0$, $\kappa_b > 0$,~\eqref{eq:positiveroot} must be taken in the limit $\kappa_a \to 0$ and reduces to $2 - 2 \cos(\sqrt{\kappa_b} t) - \sin(\sqrt{\kappa_b}t) = 0$ and $t_*(0,\kappa_b) = 2\pi /\sqrt{\kappa_b}$.
\end{remark}
With the explicit expressions for the canonical curvature of Theorem~\ref{t:curvQneq0}, we obtain, from Theorem~\ref{t:bonnetmyers2}, the following result.
\begin{theorem}\label{t:bmym}
Consider a complete, contact structure of dimension $2d+1$, with $d > 1$. Assume that there exist constants $\kappa_{1} > \kappa_{2}\geq 0$ such that, for any horizontal unit vector $X$
\begin{equation}\label{eq:bmYM}
\Ric(X) - R(X,JX,JX,X) \geq (2d-2)\kappa_{1},\qquad \|Q(X,X)\|^{2}\leq (2d-2)\kappa_{2}.
\end{equation}
Then the manifold is compact with sub-Riemannian diameter not greater than $\pi/\sqrt{\kappa_{1}-\kappa_{2}}$, and the fundamental group is finite.
\end{theorem}
\begin{proof}
From the results of Theorem~\ref{t:curvQneq0}, we have for any unit speed geodesic $\gamma(t)$
\begin{equation}
\Riccan^c_{\gamma(t)} = \Ric(\tanf) - R(\tanf,J\tanf,J\tanf,\tanf) + \frac{1}{4}h_0^2 (2d-2) - \|Q(\tanf,\tanf)\|^2,
\end{equation}
where $\tanf = \dot{\gamma}(t)$. Under our assumptions
\begin{equation}
\Riccan^c_{\gamma(t)} \geq (2d-2)(\kappa_1-\kappa_2), \qquad \all t \geq 0.
\end{equation}
The result then follows from Theorem~\ref{t:bonnetmyers2}.
\end{proof}
\begin{remark}
The Ricci tensor appearing in Eq.~\eqref{eq:bmYM} is the trace of Tanno's curvature, that is 
\begin{equation}
\Ric(X) = \sum_{i=1}^{2d}R(X,X_i,X_i,X) + R(X,X_0,X_0,X),
\end{equation}
for any orthonormal basis $X_1,\ldots,X_{2d}$ of the sub-Riemannian structure. One can check that $R(X,X_0,X_0,X)=0$ (since $\nabla X_0 = 0$ for Tanno connection). Therefore the l.h.s. of the first formula of Eq.~\eqref{eq:bmYM} is the partial trace on the $2d-2$ dimensional subspace $\spn\{X \oplus JX\}^\perp \cap \distr$.
\end{remark}
We directly obtain the following corollary for strongly pseudo-convex CR manifolds (that is, for $Q = 0$). Notice that this condition is strictly weaker than Sasakian.

\begin{corollary}\label{c:afinale}
Consider a complete, strongly pseudo-convex CR structure of dimension $2d+1$, with $d>1$, such that, for any horizontal unit vector $X$
\begin{equation}
\Ric(X) - R(X,JX,JX,X) \geq (2d-2)\kappa > 0.
\end{equation}
Then the manifold is compact with sub-Riemannian diameter not greater than $\pi/\sqrt{\kappa}$, and the fundamental group is finite.
\end{corollary}
\begin{example}\label{ex:hopf}
The classical example is the Hopf fibration 
$
S^1 \hookrightarrow \mathbb{S}^{2d+1} \xrightarrow{\pi} \mathbb{CP}^{d},
$
where the sub-Riemannian structure on $\mathbb{S}^{2n+1}$ is given by the restriction of the round metric to the orthogonal complement $\distr_x := (\ker d_x \pi)^\bot$. In this case $Q=0$ and Corollary \ref{c:afinale} applies with $\kappa=1$. The bound on the sub-Riemannian diameter (equal to $\pi$) is then sharp (see also \cite{LLZ-Sasakian}). 
\end{example}


In the 3D case (i.e. $d=1$) $Q=0$ automatically and $\Ric(X) = R(X,JX,JX,X)$. Then Theorem~\ref{t:bmym} never applies as $\kappa_1 =\kappa_2 = 0$. One can still obtain a Bonnet-Myers result from Theorem~\ref{t:bm-general}. However, as one can see from Theorem~\ref{t:curvQ=0}, the explicit statement in terms of standard tensors results in complicated expressions. This simplifies for Sasakian structures (in any dimension), where also $\tau =0$. In this case $\mathfrak{Ric}^a=0$, we apply Theorem~\ref{t:bm-general} and we obtain the following result (already proved in \cite{AAPL} for the 3D case and \cite{LLZ-Sasakian} for the general case).
\begin{corollary}
Consider a complete, Sasakian structure of dimension $2d+1$, with $d\geq 1$, such that, for any horizontal unit vector $X$
\begin{equation}
R(X,JX,JX,X) \geq \kappa > 0.
\end{equation}
Then the manifold is compact with sub-Riemannian diameter not greater than $2\pi/\sqrt{\kappa}$, and the fundamental group is finite.
\end{corollary}
For the Hopf fibration $R(X,JX,JX,X) = 4$ and we  still obtain the sharp bound on the sub-Riemannian diameter (equal to $\pi$), which agrees with Example~\ref{ex:hopf}.
\section{Computation of the canonical curvature}\label{s:computations}

We are ready to compute the canonical frame for an extremal $\lambda(t)$ of a contact sub-Rieman\-nian structure. Recall that $\dim M = 2d+1$. Fix a normalized covector such that $2H(\lambda) = 1$. Let $\gamma_\lambda(t)$ the associated unit speed geodesic, with tangent vector $\tanf =\dot\gamma_\lambda$. $X_0$ is the Reeb vector field and $h_0(\lambda) = \langle \lambda,X_0\rangle$.

\begin{theorem}[Canonical splitting]\label{t:cansplitting}
Let $\gamma_{\lambda}(t)$ be a unit speed geodesic of a contact sub-Rieman\-nian structure. Then the canonical splitting is
\begin{equation}
T_{\gamma(t)} M = S^a_{\gamma(t)} \oplus S^b_{\gamma(t)} \oplus S^c_{\gamma(t)},
\end{equation}
where, defining $X_a:= X_0 - 2Q(\tanf,\tanf)+h_0\tanf$ and $X_b:= -J\tanf$, we have:
\begin{align}
S^a_{\gamma(t)}& := \spn\{X_a\},  & \dim S^a_{\gamma(t)} & = 1, \\
S^b_{\gamma(t)}& := \spn\{X_b\},   & \dim S^b_{\gamma(t)} & = 1, \\
S^c_{\gamma(t)}& := J\tanf^\perp \cap \distr_{\gamma(t)},   & \dim S^c_{\gamma(t)} & = 2d-1.
\end{align}
where everything is computed along the geodesic. In particular, $\distr_{\gamma(t)} = S^b_{\gamma(t)} \oplus S^c_{\gamma(t)}$ and $\tanf|_{\gamma(t)} \in S^c_{\gamma(t)}$.
\end{theorem}

The directional curvature is a symmetric operator $\mathfrak{R}_{\gamma(t)} : T_{\gamma(t)} M \to  T_{\gamma(t)} M$. As in Sec.~\ref{s:invariantspaces}, for $\alpha= a,b,c$, we denote by $\mathfrak{R}_{\gamma(t)}^{\alpha\beta} : S^\alpha_{\gamma(t)} \to S^\beta_{\gamma(t)}$ the restrictions of the canonical curvature to the appropriate invariant subspace. We suppress the explicit dependence on $\gamma(t)$ from now on. 

Let $X_a$ and $X_b$ defined as above, and $\{X_{c_j}\}_{j=2}^{2d}$ be an orthonormal frame for $S^c$. Without loss of generality we assume $X_{c_{2d}} = \tanf$. Then $S^c = \spn\{X_{c_2},\ldots,X_{c_{2d-1}}\}\oplus \spn\{\tanf\}$. 

\begin{theorem}[Canonical curvature, case $Q=0$]\label{t:curvQ=0} In terms of the above frame, 
\begin{align}
\mathfrak{R}^{aa} & = 
g((\nabla_\tanf^2\tau)(\tanf),J\tanf) -5 h_0g((\nabla_\tanf\tau)(\tanf),\tanf)- 2  g(\tau(\tanf),\tanf)^2- 6 h_0^2 g(\tau(\tanf),J\tanf) \\ & \quad  - 2\|\tau(\tanf)\|^2 - g((\nabla_{X_0}\tau)(\tanf),\tanf), \\
\mathfrak{R}^{ab} & = 0, \\
\mathfrak{R}^{ac} & = R(\tanf, X_0, X_j, \tanf)+ g((\nabla_\tanf \tau)(\tanf),X_j)	 + 2h_0g(\tau(\tanf),JX_j) \\
& \quad - [ g((\nabla_\tanf \tau)(\tanf),\tanf)	 + 2h_0g(\tau(\tanf),J\tanf)]g(\tanf,X_j), \\
\mathfrak{R}^{bb}  & = R(\tanf,J\tanf,J\tanf,\tanf) - 3g(\tau(\tanf),J\tanf) + h_0^2, \\
\mathfrak{R}^{bc} & = -R(\tanf,J\tanf,X_j,\tanf) +  g(\tau(\tanf),X_j) - g(\tau(\tanf),\tanf)g(\tanf,X_j), \\
\mathfrak{R}^{cc} & =  \mathcal{S}[R(\tanf,X_{i},X_{j},\tanf)] +\frac{1}{4}h_0^2 \left[g(X_i,X_j) -g(X_i,\tanf )g(X_j,\tanf )\right],
\end{align}
where $i,j=2,\ldots,2d$ and $\mathcal{S}$ denotes symmetrisation. The canonical Ricci curvatures are
\begin{align}
\mathfrak{Ric}^{a} & =  
g((\nabla_\tanf^2\tau)(\tanf),J\tanf) -5 h_0g((\nabla_\tanf\tau)(\tanf),\tanf)- 2  g(\tau(\tanf),\tanf)^2- 6 h_0^2 g(\tau(\tanf),J\tanf) \\ &\quad  - 2\|\tau(\tanf)\|^2 - g((\nabla_{X_0}\tau)(\tanf),\tanf), \\
\mathfrak{Ric}^{b}  & =  R(\tanf,J\tanf,J\tanf,\tanf)  - 3g(\tau(\tanf),J\tanf) + h_0^2, \\
\mathfrak{Ric}^{c} & =  \mathrm{Ric}(\tanf) - R(\tanf,J\tanf,J\tanf,\tanf) + \frac{1}{4}h_0^2(2d-2).
\end{align}
\end{theorem}
\begin{theorem}[Canonical curvature, general $Q$]\label{t:curvQneq0} In terms of above frame
\begin{align}
\mathfrak{R}^{bb}  & =  R(\tanf,J\tanf,J\tanf,\tanf) + 3 \|Q(\tanf,\tanf)\|^2 - 3g(\tau(\tanf),J\tanf) + h_0^2, \\
\mathfrak{R}^{bc} & = -R(\tanf,J\tanf,X_j,\tanf) +  g(\tau(\tanf),X_j) - g(\tau(\tanf),\tanf)g(\tanf,X_j)  + 3g((\nabla_\tanf Q)(\tanf,\tanf),X_j)\\ & \quad +8h_0g(Q(\tanf,\tanf),JX_j),\\
\mathfrak{R}^{cc} & = \mathcal{S}[R(\tanf,X_{i},X_{j},\tanf)]+h_0 \mathcal{S}[g(\tanf,Q(X_{j},X_{i}))] +\frac{1}{4}h_0^2 g(X_i,X_j) \\
& \quad - \frac{1}{4}g( X_i,h_0\tanf - 2Q(\tanf,\tanf))g(X_j,h_0\tanf -2Q(\tanf,\tanf)),
\end{align}
where $i,j=2,\ldots,2d$ and $\mathcal{S}$ denotes symmetrisation. The canonical Ricci curvatures are
\begin{align}
\mathfrak{Ric}^{b}  = & R(\tanf,J\tanf,J\tanf,\tanf) + 3 \|Q(\tanf,\tanf)\|^2 - 3g(\tau(\tanf),J\tanf) + h_0^2, \\
\mathfrak{Ric}^{c}= & \mathrm{Ric}(\tanf) - R(\tanf,J\tanf,J\tanf,\tanf) + \frac{1}{4}h_0^2(2d-2) - \|Q(\tanf,\tanf)\|^2.
\end{align}
\end{theorem}
\begin{remark}
As in the Riemannian setting, there is no curvature in the direction of the geodesic, that is $\mathfrak{R}\tanf = 0$.
\end{remark}
The rest of the section is devoted to the proofs of Theorems~\ref{t:cansplitting},~\ref{t:curvQ=0} and~\ref{t:curvQneq0}.

\subsection{Lifted frame}

Only in this subsection, $M$ is an $n$-dimensional manifold, since the construction is general. We define a local frame on $T^*M$, associated with the choice of a local frame $X_1,\dots,X_n$ on $M$. All our considerations are local, then we assume that the frame is globally defined. For $\alpha = 1,\dots,n$ let $h_\alpha :T^*M \to \mathbb{R}$ be the linear-on-fibers function defined by $ h_\alpha(\lambda) := \langle \lambda, X_\alpha\rangle$. The action of derivations on $T^*M$ is completely determined by the action on affine functions, namely functions $a \in C^\infty(T^*M)$ such that $a(\lambda) = \langle \lambda, Y \rangle + \pi^* g$ for some $Y\in \Gamma(TM)$, $g\in C^\infty(M)$. Then, we define the \emph{lift of a field} $X \in \Gamma(TM)$ as the field $\wt{X} \in  \Gamma(T(T^*M))$ such that $\wt{X}(h_\alpha) = 0$ for $\alpha=1,\dots,n$ and $\wt{X}(\pi^*g) = X(g)$. This, together with Leibniz's rule, characterize the action of $\wt{X}$ on affine functions, and completely define $\wt{X}$. Observe that $\pi_* \wt{X} = X$. On the other hand, we define the (vertical) fields $\partial_{h_\alpha}$ such that $\partial_{h_\alpha} (\pi^* g) = 0$, and $\partial_{h_\alpha} (h_\beta) = \delta_{\alpha\beta}$. 
It is easy to check that $\{ \partial_{h_\alpha}, \wt{X}_\alpha\}_{\alpha=1}^{n}$ is a local frame on $T^*M$. We call such a frame the \emph{lifted frame}.
\begin{remark}
Let $(q_1,\ldots,q_n)$ be local coordinates on $M$, and let $X_\alpha = \partial_{q_\alpha}$. In this case, our construction gives the usual frame $\{\partial_{q_\alpha},\partial_{p_\alpha}\}_{\alpha=1}^n$ on $T^*M$ associated with the dual coordinates $(q_1,\ldots,q_n,p_1,\ldots,p_n)$ on $T^*M$. The construction above is then a generalization of this classical construction to non-commuting local frames.
\end{remark}
Any fixed lifted frame defines a splitting
\begin{equation}
T_\lambda (T^*M) = \spn_\lambda\{\wt{X}_1,\ldots,\wt{X}_n\} \oplus \spn_\lambda\{\partial_{h_1},\ldots,\partial_{h_n}\}.
\end{equation}
Since $\spn_\lambda\{\partial_{h_1},\ldots,\partial_{h_n}\} = \ker\pi_*|_\lambda$ is the vertical subspace of $T_\lambda(T^*M)$, we associate, with any vector $\xi \in T_\lambda(T^*M)$ its \emph{horizontal part} $\xi^h \in T_x M$ and its \emph{vertical part} $\xi^v \in T_x^*M$:
\begin{equation}
\xi^h := \pi_* \xi, \qquad \xi^v := \xi - \wt{\pi_* \xi},
\end{equation}
where here, as we will always do in the following, we used the canonical identification $\ker \pi_*|_\lambda = T_\lambda (T^*_x M) \simeq T_x^*M$, where $x = \pi(\lambda)$. We stress that the concept of vertical part introduced here makes sense w.r.t. a fixed frame, has no invariant meaning, and is only a tool for computations.

Let $\nu_1,\dots,\nu_{n}$ be the dual frame of $X_1,\dots,X_{n}$, i.e. $\nu_\alpha(X_\beta) = \delta_{\alpha\beta}$. In the following, we use the notation $\partial_\alpha:=\partial_{h_\alpha}$ and we suppress the tilde from $\wt{X}$; the meaning will be clear from the context. Then the symplectic form is written as
\begin{equation}\label{eq:sympl}
\sigma =  \sum_{\alpha=1}^n dh_\alpha\wedge \pi^* \nu_\alpha + h_\alpha \pi^* d \nu_\alpha.
\end{equation}
We say that a vector $V_\lambda \in T_{\lambda}(T^*M)$ is \emph{vertical} if $\pi_* V = 0$. In this case $V_\lambda \in T_{\lambda}(T_{\pi(\lambda)}^*M) \simeq T_{\pi(\lambda)}^*M$ can be identified with a covector. We have the following useful lemma.
\begin{lemma}
Let $V,W \in T_\lambda(T^*M)$. If $V$ is vertical, then
\begin{equation}
\sigma_\lambda(V,W) = \langle V , \pi_* W\rangle.
\end{equation}
\end{lemma}
Using Eq.~\eqref{eq:sympl}, the Hamiltonian vector field associated with $h_\alpha$ has the form
\begin{equation}
\vec{h}_\alpha = X_\alpha + \sum_{\delta=1}^n c_{\alpha\beta}^\delta h_\delta \partial_\beta, \qquad \alpha,\beta =1,\ldots,n.
\end{equation}
Since $H= \frac{1}{2}\sum_{i=1}^k h_i^2$, then $\vec{H} = \sum_{i=1}^k h_i \vec{h}_i$, where $k= \rank \distr$.

\subsection{Some useful formulas}\label{s:Tannocoord}
In what follows $\nabla$ will always denote Tanno connection. Let $X_1,\ldots,X_{2d}$ be an orthonormal frame for $\distr$ and $X_0$ be the Reeb field. The \emph{structural functions} are defined by
\begin{equation}
[X_\alpha,X_\beta] = \sum_{\delta=0}^{2d} c_{\alpha\beta}^\delta X_\delta, \qquad \alpha,\beta = 0,\ldots,2d.
\end{equation}
Notice that $c_{\alpha 0}^0 = \omega([X_\alpha,X_0]) = d\omega (X_\alpha,X_0) =0$. We define ``horizontal'' Christoffel symbols:
\begin{equation}
\Gamma_{ij}^k := \frac{1}{2}\left( c_{ij}^k + c_{ki}^j + c_{kj}^i \right), \qquad i,j,k=1,\ldots,2d.
\end{equation}
Notice that $\Gamma_{ij}^k + \Gamma_{ik}^j=0$ and that one can recover some of the structural functions with the relation
\begin{equation}
c_{ij}^k = \Gamma_{ij}^k - \Gamma_{ji}^k.
\end{equation}
In terms of the structural functions, we have, for $i,j,k=1,\ldots,2d$
\begin{equation}
\nabla_{X_i} X_j = \sum_{k=1}^{2d} \Gamma_{ij}^k X_k, \qquad \nabla_{X_0} X_i = \frac{1}{2}\sum_{k=1}^{2d}(c_{k0}^i - c_{i0}^k) X_k,
\end{equation}
\begin{equation}
\tau(X_i) = \frac{1}{2}\sum_{k=1}^{2d}(c_{k0}^i + c_{i0}^k) X_k, \qquad T(X_j,X_k) = - c_{jk}^0 X_0, \qquad JX_i = \sum_{j=1}^{2d} c_{ij}^0 X_j.
\end{equation}

\begin{lemma}
Let $\lambda \in T^*M$ an initial covector, associated with the geodesic $\gamma_\lambda(t) = \pi(\lambda(t))$. In terms of Tanno covariant derivative, along the geodesic $\gamma_\lambda(t)$, we have
\begin{align}
\nabla_\tanf \tanf & = h_0 J\tanf, \\
\nabla_\tanf (J\tanf) & = -h_0\tanf + Q(\tanf,\tanf), \\
\nabla_\tanf h_0 & = g(\tau(\tanf),\tanf),
\end{align}
where $h_0 = h_0(\lambda(t)) = \langle\lambda(t),X_0\rangle$ and $\tanf:=\dot{\gamma}_\lambda$.
\end{lemma}
\begin{proof}
Along the geodesic, $\tanf|_{\gamma(t)} = \sum_{i=1}^{2d} h_i(t) X_i|_{\gamma(t)}$, where $h_i(t) = h_i(\lambda(t)) = \langle \lambda(t),X_i\rangle$. Then
\begin{align}
\nabla_\tanf \tanf & = \sum_{i=1}^{2d} \dot{h}_i X_i + \sum_{i,j=1}^{2d} h_j h_i \nabla_{X_i} X_j = \sum_{i=1}^{2d} \{H,h_i\} X_i + \sum_{i,j,k=1}^{2d} h_j h_i \Gamma_{ij}^k X_k  \\ 
& = \sum_{\alpha=0}^{2d} \sum_{i,j=1}^{2d} h_j c_{ji}^\alpha h_\alpha X_i + \sum_{i,j,k=1}^{2d} h_j h_i \Gamma_{ij}^k X_k = h_0 \sum_{j=1}^{2d} h_j c_{ji}^0 X_i= \sum_{j=1}^{2d} h_0 h_j JX_j = h_0 J\tanf.
\end{align}
where we used the Hamilton's equation for normal geodesics. The second identity follows from the first and the definition of $Q$. For the third identity, we use again Hamilton's equations:
\begin{equation}
\nabla_\tanf h_0 = \dot{h}_0 = \{ H, h_0\} = \sum_{i,j=1}^{2d} h_i c_{i0}^j h_j = \frac{1}{2} \sum_{i,j=1}^{2d} h_i h_j (c_{i0}^j + c_{j0}^i) = g(\tau(\tanf),\tanf).\qedhere
\end{equation}
\end{proof}

\begin{lemma}\label{lem:identities}
Let $X, Y,Z $ be vector fields. Then 
\begin{itemize}
\item[a)] $Q(X,Y)$ is horizontal,
\item[b)] $Q(JX,Y) = -J Q(X,Y)$,
\item[c)] $Q(X,Y) \perp X$,
\item[d)] $g(X,Q(Y,Z)) = - g(Y,Q(X,Z))$ (skew-symmetry w.r.t. the first two arguments),
\item[e)] $Q(X,Y) \perp JX$,
\item[f)] $\mathfrak{S}g(\Tor(X,Y),JZ) =0$, where $\mathfrak{S}$ denotes the cyclic sum,
\item[g)] $\mathfrak{S}g(Y,Q(Z,X)) = 0$,
\item[h)] $Q(Y,JY) = Q(JY,Y) = -JQ(Y,Y)$,
\item[i)] $\trace Q=0$.
\end{itemize}
\end{lemma}
\begin{proof}
To prove a), observe that, since $JX_0 = 0$, it follows that $Q(X_0,Y) = 0$ for any $Y$. Then we can assume w.l.o.g. that $X$ is horizontal. Clearly $JX$ is horizontal too. Then $\omega(JX) = 0$. Therefore, since $\omega$ is parallel for Tanno connection, $0 = \omega(\nabla_Y(J)X + J \nabla_Y X) = \omega(Q(X,Y))$, where we used the fact that $\nabla_Y X$ is horizontal. To prove b), observe first that if $X= X_0$ the identity is trivially true. Then, w.l.o.g. we assume that $X$ is horizontal. Consider the covariant derivative of the identity $J^2 X = -X$. Then we obtain $(\nabla_Y J)(JX) + J (\nabla_Y J) X + J^2 \nabla_Y X = -\nabla_Y X$, which implies b), by definition of $Q$. To prove c), observe that $X \perp JX$ by definition. Then $g(JX,X) = 0$. By taking the covariant derivative we obtain $g(Q(X,Y),X) + g(J\nabla_Y X, X) + g(JX, \nabla_Y X)= 0$, which implies c). To prove d), observe that the identity is trivially true when $X = X_0$ or $Y = X_0$. Then we may assume w.l.o.g. that $X, Y$ are horizontal. Then we take the covariant derivative in the direction $Z$ of the identity $g(JX,Y) = -g(X,JY)$, and we obtain the result. Point e) follows from d) and b). 
To prove f), we have
\begin{equation}
\mathfrak{S}g(\Tor(X,Y),JZ) = g(\Tor(X,Y),JZ) + g(\Tor(Z,X),JY)+ g(\Tor(Y,Z),JX).
\end{equation}
If $X,Y,Z \in \distr$, then all the terms are trivially zero. Then assume w.l.o.g. that $X = X_0$. Then
\begin{align}
\mathfrak{S}g(\Tor(X_0,Y),JZ) & = g(\Tor(X_0,JY),Z) - g(\Tor(X_0,JZ),Y)+ g(\Tor(Y,Z),\cancel{JX_0}) \\
& = g(\tau(JY),Z) - g(\tau(JZ),Y) = 0,
\end{align}
where we used the fact that $\tau$ is symmetric and $\tau \circ J +J \circ \tau = 0$. To prove g), consider the following identity for the differential of a $2$-forms $\alpha$:
\begin{equation}
d\alpha(X,Y,Z) = \mathfrak{S}(X(\alpha(Y,Z)) - \mathfrak{S}\alpha([X,Y],Z).
\end{equation}
We apply it to the two form $\alpha = d\omega$. Since $d\alpha =d^2\omega =0$, we have
\begin{align}
0 & =  d^2\omega = \mathfrak{S}\left[ X(d\omega(Y,Z)) - d\omega([X,Y],Z)\right] =\mathfrak{S}\left[ X(g(Y,JZ)) -g([X,Y],JZ)\right]  \\
& = \mathfrak{S}\left[ \cancel{g(\nabla_X Y, JZ)} + g(Y,Q(Z,X)) + g(Y,J\nabla_X Z) - g(\cancel{\nabla_X Y} - \nabla_Y X,JZ)\right] \\
&  \quad + \xcancel{\mathfrak{S}g(\Tor(X,Y),JZ)} \\
& = \mathfrak{S}\left[g(Y,Q(Z,X)) - \cancel{g(\nabla_X Z),JY)} + \cancel{g(\nabla_Y X ,JZ)}\right].
\end{align}
where we used the definition of $J$ and property f). Point h) follows from g), e), d), and b). 
For i) observe that $Q(X_{0},X_{0})=0$. Then, for every orthonormal frame $X_{1},\ldots,X_{2d}$ of $\distr$, we have
\bqn
\trace Q= \sum_{i=1}^{2d}Q(X_{i},X_{i})=-\sum_{i=1}^{2d}J^{2}Q(X_{i},X_{i})=\sum_{i=1}^{2d}JQ(JX_{i},X_{i})=-\sum_{i=1}^{2d}Q(JX_{i},JX_{i})=-\trace Q,
\eqn
where we used properties a), h) and b).
\end{proof}
Tanno connection has torsion. Then the curvature tensor 
\begin{equation}
R(X,Y,Z,W):=g(\nabla_X\nabla_Y Z - \nabla_Y \nabla_X Z - \nabla_{[X,Y]}Z,W)
\end{equation} 
is not symmetric w.r.t.\ to exchange of the first and second pair of arguments. 
\begin{lemma}\label{l:propcurv}
For any $X,Y,Z\in \distr$ we have
\begin{equation}\label{eq:id1}
\begin{split}
R(X,Z,Y,X) & = R(X,Y,Z,X) + g(X,J Z)g(X,\tau(Y)) + g(Z,JY)g(X,\tau(X)) \\
& \quad + g(Y,JX)g(X,\tau(Z)).
\end{split}
\end{equation}
In particular, if $Y,Z \in JX^\perp\cap\distr$ we have
\begin{equation}\label{eq:id1b}
R(X,Z,Y	,X) = R(X,Y,Z,X) + g(Z,JY)g(X,\tau(X)).
\end{equation}
Moreover, for Tanno's tensor we have
\begin{equation}\label{eq:id2}
\mathcal{S}g(X,Q(X_i,X_j))=  \frac{1}{2} g(X_j,Q(X_i,X))-  g(X_j,Q(X,X_i)),
\end{equation}
where $\mathcal{S}$ denotes the symmetrization w.r.t.\ the displayed indices. 
\end{lemma}
\begin{proof}
The first identity is a consequence of first Bianchi identity for connections with torsion:
\begin{equation}
\mathfrak{S}(R(X,Y)Z) = \mathfrak{S}(\Tor(\Tor(X,Y),Z)) + (\nabla_X\Tor)(Y,Z),
\end{equation}
where $\mathfrak{S}$ denotes the cyclic sum. If $X,Y,Z \in \distr$, for Tanno connection we obtain
\begin{align}
\mathfrak{S}(R(X,Y)Z)) & =  \mathfrak{S}\left[g(X,JY)\tau(Z) + \nabla_X(\Tor(Y,Z)) - \Tor(\nabla_X Y,Z) - \Tor(Y,\nabla_X Z)\right]  \\
& =  \mathfrak{S}\left[g(X,JY)\tau(Z) + \nabla_X(g(Y,JZ))X_0 - g(\nabla_X Y,JZ)X_0 - g(Y,J\nabla_X Z) X_0\right]  \\
& =  \mathfrak{S}\left[g(X,JY)\tau(Z)\right] + \cancel{\mathfrak{S} g(Y,Q(Z,X))},
\end{align}
where we used property g) of Tanno's tensor. If we take the scalar product with $X$, we get
\begin{equation}
R(X,Y,Z,X) + R(Y,Z,X,X) +R(Z,X,Y,X) = g(X,\mathfrak{S}g(X,JY)\tau(Z)).
\end{equation}
Tanno's curvature is still skew symmetric in the first and last pairs of indices, and we get
\begin{equation}
\begin{split}
R(X,Y,Z,X) - R(X,Z,Y,X) & = g(X,JY)g(X,\tau(Z)) + g(Y,JZ)g(X,\tau(X)) \\
& \quad + g(Z,JX)g(X,\tau(Y)).
\end{split}
\end{equation}
This proves~\eqref{eq:id1}. For~\eqref{eq:id2}, using property g) of Tanno's tensor, we get
\begin{equation}
\frac{1}{2} g(X_j,Q(X_i,X))-  g(X_j,Q(X,X_i)) = \frac{1}{2}\left[g(X,Q(X_i,X_j)) + g(X,Q(X_j,X_i))\right]. \qedhere
\end{equation}
\end{proof}

\subsection{Computation of the curvatures}

We choose a conveniently adapted frame along the geodesic, of which we will consider the lift. 
\begin{definition} A local frame $X_0,X_1,\ldots,X_{2d}$ is \emph{adapted} if
\begin{itemize}
\item $X_0$ is Reeb vector field,
\item $X_{2d}$ is an horizontal extension of $\dot{\gamma}_\lambda(t)$, namely $X_{2d} \in \distr$ and $\dot\gamma_\lambda(t) = X_{2d}(\gamma_\lambda(t))$,
\item $X_1 = -JX_{2d}$,
\item $X_2,\ldots,X_{2d-1} \in \mathrm{span}(X_{2d},JX_{2d})^\perp$ are orthonormal.
\end{itemize}
\end{definition}
In particular, $X_0,X_1,\ldots,X_{2d}$ is an orthonormal frame for the Riemannian extension of $g$. From now on the last index $2d$ plays a different role, since it is constrained to be an horizontal extension of the tangent vector of the fixed geodesic. 

\begin{definition} An adapted frame $X_0,X_1,\ldots,X_{2d}$ is \emph{parallel transported} if
\begin{equation}\label{eq:partrans2}
\nabla_{\tanf} X_j = \frac{1}{2}(h_0 J X_j + a_j JX_{2d}), \qquad j =2,\ldots,2d
\end{equation}
along the geodesic $\gamma_\lambda(t)$, where
\begin{equation}
a_j := g(h_0 X_{2d}- 2Q(X_{2d},X_{2d}),X_j).
\end{equation}
\end{definition}
As one can readily check, an adapted, parallel transported frame is unique up to constant rotation of $X_2,\ldots,X_{2d-1}$. Notice that $a_j$ and $h_0$ are well defined functions on the curve $\gamma_\lambda(t)$ once the initial covector $\lambda$ is fixed.

\subsubsection{The algorithm}

To compute the canonical frame and curvatures, we follow the general algorithm developed in \cite{lizel}. The elements of the canonical frame and the curvature matrices will be computed in the following order:
\begin{equation}
E_a \to E_b \to F_b \to E_{c} \to F_{c} \to R_{bb}, R_{c c}, R_{bc} \to F_a \to R_{aa},R_{ac}.
\end{equation}
We choose an adapted parallel transported frame $X_0,X_1,\ldots,X_{2d}$ as defined above, and the associated lifted frame $\{X_\alpha,\partial_{\alpha}\}_{\alpha=0}^{2d}$. Along the extremal, we have 
\begin{equation}
h_i = \langle\lambda,X_i\rangle = g(X_{2d},X_i) = \delta_{i,2d}, \qquad i=1,\ldots,2d.
\end{equation}
We have the following simplifications for some structural functions, for $\alpha=0,\ldots,2d$:
\begin{align}
c_{1,\alpha}^0 & = \omega([X_1,X_\alpha]) = -d\omega(X_1,X_\alpha) = -g(X_1,JX_\alpha) = g(JX_{2d},JX_\alpha) = \delta_{2d,\alpha}, \\
c_{2d,\alpha}^0 & = \omega([X_{2d},X_\alpha]) = -d\omega(X_{2d},X_\alpha) = -g(X_{2d},JX_\alpha) = - g(X_1,X_\alpha) = \delta_{1,\alpha}.
\end{align}
From here, repeated Latin indices are summed from $1$ to $2d$, and Greek ones from $0$ to $2d$.

\subsubsection{Preliminary computations}
The Lie derivative w.r.t. $\vec{H}$ is denoted with a dot, namely $\dot{E}_{\alpha} = [\vec{H},E_{\alpha}]$. We compute it for the basic elements of the lifted frame $\{X_\alpha,\partial_\alpha\}_{\alpha=0}^{2d}$. Here, we use the shorthand $\nabla_\alpha = \nabla_{X_\alpha}$ for $\alpha=0,\ldots,2d$.
\begin{lemma}\label{l:prelcomps}
For any adapted frame (not necessarily parallel transported), we have
\begin{align}
\dot{\partial}_0 & = \partial_1, \\
\dot{\partial}_i & = -X_i - 2g(\tau(X_{2d}),X_i)\partial_0 + h_0 g(JX_j,X_i) \partial_j - g(\nabla_{2d} X_j,X_i) \partial_j - g(\nabla_i X_j,X_{2d}) \partial_j, \\
\dot{X}_i & = \nabla_{2d} X_i - \nabla_i X_{2d} -\delta_{i1} X_0 - \nabla_i g(\tau(X_{2d}),X_{2d}) \partial_0 - \nabla_i g(X_{2d},\nabla_{2d} X_j) \partial_j, \\
\dot{X}_0 & = -\nabla_0 X_{2d} + \tau(X_{2d}) - \nabla_0 g(\tau(X_{2d}),X_{2d}) \partial_0 - \nabla_0 g(X_{2d},\nabla_{2d} X_j) \partial_j.
\end{align}
\end{lemma}
\begin{proof}
The proof is a routine computation, using the properties of the adapted frame.
\begin{equation}
\dot{\partial}_0 = [\vec{H},\partial_0] = [h_i\vec{h}_i,\partial_0] = h_i [X_i + c_{i \alpha}^\beta h_\beta \partial_\alpha,\partial_0] = c_{2d,\alpha}^0 \partial_\alpha = \partial_1.
\end{equation}
For $i=1,\ldots,2d$, we have
\begin{align}
\dot{\partial}_i  = [\vec{H},\partial_i] & = [h_j\vec{h}_j,\partial_i] = -\vec{h}_i - c_{2d,\alpha}^i \partial_\alpha \\
& = - X_i - c_{i\alpha}^\beta h_\beta \partial_\alpha  - c_{2d,\alpha}^i \partial_\alpha\\
& = -X_i - c_{i0}^{2d} \partial_0 - c_{ij}^{2d} \partial_j- c_{ij}^0 h_0 \partial_j - c_{2d,0}^i \partial_0 - c_{2d,j}^i \partial_j  \\
& = -X_i - (c_{i0}^{2d} + c_{2d,0}^i)\partial_0 - c_{ij}^0 h_0 \partial_j - (c_{2d,j}^i + c_{ij}^{2d} )\partial_j  \\
& = -X_i - 2g(\tau(X_{2d}),X_i)\partial_0 - h_0 g(X_j,JX_i) \partial_j - (\Gamma_{2d,j}^i +\Gamma_{ij}^{2d})\partial_j  \\
& = -X_i - 2g(\tau(X_{2d}),X_i)\partial_0 + h_0 g(JX_j,X_i) \partial_j - g(\nabla_{2d} X_j,X_i) \partial_j - g(\nabla_i X_j,X_{2d}) \partial_j.
\end{align}
Moreover, for $\mu=0,\ldots,2d$
\begin{align}
\dot{X}_\mu  = [\vec{H},X_\mu] & = [h_i \vec{h}_i,X_\mu] = [\vec{h}_{2d}, X_\mu] \\
& = [X_{2d} + c_{2d,\alpha}^\beta h_\beta \partial_\alpha,X_\mu]  \\
& = [X_{2d},X_\mu] - X_\mu(c_{2d,\alpha}^\beta) h_\beta \partial_\alpha \\
& = [X_{2d},X_\mu] - X_\mu(c_{2d,0}^{2d}) \partial_0 - X_\mu(c_{2d,j}^{2d}) \partial_j - X_\mu(c_{2d,j}^0)h_0 \partial_j \\
& = \nabla_{2d} X_\mu - \nabla_\mu X_{2d} - \Tor(X_{2d},X_\mu) - X_\mu(g(\tau(X_{2d}),X_{2d})) \partial_0 \\
& \quad - X_\mu(g(\nabla_{2d} X_j,X_{2d})) \partial_j,
\end{align}
and we obtain the result using the properties of the torsion of Tanno connection. 
\end{proof}

\subsubsection*{Step 1: $E_a$}
The element $E_a$ is uniquely determined by the following conditions:
\begin{itemize}
\item[(i)] $\pi_* E_a = 0$,
\item[(ii)] $\pi_* \dot{E}_a = 0$,
\item[(iii)] $2H(\dot{E}_a) = 1$ (equivalent to $\sigma(\ddot{E}_a,\dot{E}_a) = 1$).
\end{itemize}
The first condition implies $E_a = v_\alpha \partial_\alpha$ for some smooth function $v_\alpha$ along the extremal. The second condition, using Lemma~\ref{l:prelcomps}, implies $E_a = v_0 \partial_0$. To apply the third condition, observe that
\begin{equation}
\dot{E}_a = \dot{v}_0 \partial_0 + v_0 \dot{\partial}_0 = \dot{v}_0 \partial_0 + v_0 \partial_1. 
\end{equation}
By standard identifications, $2H(\dot{E}_{a}) = v_0^2 \|X_1\|^2 = v_0^2$, and $v_0 = \pm 1$ (we choose the ``+'' sign).

\subsubsection*{Step 2: $E_b$}
$E_a$ directly determines $E_b$ through the structural equation:
\begin{equation}
E_b = \dot{E}_a = \dot\partial_0 = \partial_1.
\end{equation}
\subsubsection*{Step 3: $F_b$}
$E_b$ directly determines $F_b$ through the structural equation:
\begin{align}
F_b & = \dot{E}_b = -\dot{\partial}_1 \\
& = X_1 + 2g(\tau(X_{2d}),X_1)\partial_0 + h_0 g(JX_1, X_j) \partial_j - g(\nabla_{2d} X_1,X_j) \partial_j - g(\nabla_1  X_{2d},X_j) \partial_j \\
& = X_1 + 2g(\tau(X_{2d}),X_1)\partial_0 + \cancel{h_0 g(X_{2d}, X_j) \partial_j} + g(Q(X_{2d},X_{2d}) - \cancel{h_0 X_{2d}} ,X_j) \partial_j\\
& \quad  - g(\nabla_1  X_{2d},X_j) \partial_j \\
& = X_1 + 2g(\tau(X_{2d}),X_1)\partial_0 +  g(Q(X_{2d},X_{2d})-\nabla_1  X_{2d} ,X_j) \partial_j.
\end{align}
Where we used the equation $\nabla_{2d} X_{2d} = h_0 JX_{2d}$.

\subsubsection*{Intermediate step: $\dot{F}_b$}
We now compute $\dot{F}_b$. We use Leibniz's rule, Lemma~\ref{l:prelcomps} and we obtain:
\begin{align}
\dot{F}_b & =  \dot{X}_1 + 2\nabla_{2d} g(\tau(X_{2d}),X_1)\partial_0 + 2 g(\tau(X_{2d}),X_1)\dot{\partial}_0 +  \nabla_{2d} g(Q(X_{2d},X_{2d})-\nabla_1  X_{2d} ,X_j) \partial_j \\ 
& \quad + g(Q(X_{2d},X_{2d})-\nabla_1  X_{2d} ,X_j) \dot\partial_j \\
& = \nabla_{2d} X_1 - \nabla_1 X_{2d} - X_0 - \nabla_1 g(\tau(X_{2d}),X_{2d}) \partial_0 - \nabla_1 g(X_{2d},\nabla_{2d} X_j) \partial_j \\ 
& \quad + 2 \nabla_{2d} g(\tau(X_{2d}),X_1) \partial_0  + 2g(\tau(X_{2d}),X_1)\partial_1 + \nabla_{2d} g(Q(X_{2d},X_{2d})-\nabla_1 X_{2d},X_j)\partial_j \\
& \quad - Q(X_{2d},X_{2d}) + \nabla_1 X_{2d} -2g(\tau(X_{2d}),Q(X_{2d},X_{2d}) -\nabla_1 X_{2d})\partial_0 \\
& \quad + h_0 g(J X_j,Q(X_{2d},X_{2d})-\nabla_1 X_{2d})\partial_j - g(\nabla_{2d} X_j,Q(X_{2d},X_{2d}) -\nabla_1 X_{2d})\partial_j \\
& \quad - g(\nabla_{Q(X_{2d},X_{2d}) - \nabla_1 X_{2d}} X_j, X_{2d}) \partial_j \\
& = \nabla_{2d} X_1 - X_0 - Q(X_{2d},X_{2d})+  2 \nabla_{2d} g(\tau(X_{2d}),X_1)\partial_0- \nabla_1 g(\tau(X_{2d}),X_{2d})\partial_0  \\ 
& \quad -2g(\tau(X_{2d}),Q(X_{2d},X_{2d})-\nabla_1 X_{2d})\partial_0  + 2g(\tau(X_{2d}),X_1)\partial_1 - \nabla_1 g(X_{2d},\nabla_{2d} X_j) \partial_j \\ 
& \quad + \nabla_{2d} g(Q(X_{2d},X_{2d})-\nabla_1 X_{2d},X_j)\partial_j  + h_0 g(J X_j,Q(X_{2d},X_{2d})-\nabla_1 X_{2d})\partial_j  \\ 
&\quad  - g(\nabla_{2d} X_j,Q(X_{2d},X_{2d})-\nabla_1 X_{2d})\partial_j   - g(\nabla_{Q(X_{2d},X_{2d}) - \nabla_1 X_{2d}} X_j, X_{2d}) \partial_j.
\end{align}

\subsubsection*{Step 4: $E_{c_j}$}
The elements $E_{c_j}$, for $j=2,\ldots,2d$ are uniquely determined by the conditions:
\begin{itemize}
\item[(i)] They generate the skew-orthogonal complement of $\spn\{E_a,E_b,F_b,\dot{F}_b\}$,
\item[(ii)] Darboux property: $\sigma(\dot{E}_{c_i},E_{c_j}) = \delta_{ij}$,
\item[(iii)] Their derivative $F_{c_j}=-\dot{E}_{c_j}$ generate an isotropic subspace (or, equivalently, assuming the two points above, $\pi_*\ddot{E}_{c_j} = 0$).
\end{itemize}
Since $\pi_* E_{c_j} =0$ we have
\begin{equation}
E_{c_j} = \alpha_{ij}\partial_i + \beta_j \partial_0, \qquad j=2,\ldots,2d,
\end{equation}
for some smooth functions $\alpha_{ij}$ and $\beta_j$ along the extremal. Observe that $\alpha$ is a $2d\times (2d-1)$ matrix. In order to apply condition (i) we compute $\pi_*\dot{F}_b$. From the previous step we get
\begin{equation}
\pi_* \dot{F}_b =  \nabla_{2d} X_1 -X_0 - Q(X_{2d},X_{2d}) = -2Q(X_{2d},X_{2d}) +h_0 X_{2d} - X_0.
\end{equation}
Then we apply (i) and we obtain:
\begin{equation}
0 = \sigma(F_b,E_{c_j}) = -\langle E_{c_j}, \pi_* F_b\rangle = -\langle E_{c_j}, X_1\rangle = -\alpha_{1j}.
\end{equation}
Thus the first row of $\alpha$ is zero. Moreover
\begin{align}
0 = \sigma(\dot{F}_b, E_{c_j}) = -\langle E_{c_j}, \pi_* \dot{F}_b\rangle & = \langle E_{c_j},X_0 + 2Q(X_{2d},X_{2d}) - h_0 X_{2d}\rangle \\
& = \beta_j + \alpha_{ij} g(2Q(X_{2d},X_{2d})-h_0X_{2d},X_i).
\end{align}
Denoting $a_i:= g(h_0 X_{2d} - 2Q(X_{2d},X_{2d}),X_i)$, for all $i=1,\ldots,2d$ we get
\begin{equation}
E_{c_j} = \alpha_{ij} (\partial_i + a_i \partial_0), \qquad j =2,\ldots,2d.
\end{equation}
Condition (iii) gives
\begin{align}
0 = \pi_* \ddot{E}_{c_j} & = \alpha_{ij}\pi_*\left(\ddot{\partial}_i + \ddot{a}_i \partial_0 + 2\dot{a}_i \dot{\partial}_0 + a_i \ddot{\partial}_0\right) + 2\dot{\alpha}_{ij}\pi_*\left(\dot{\partial}_i +\dot{a}_i\partial_0 + a_i \dot\partial_0\right)  \\
& = \alpha_{ij}\left(-\nabla_{2d} X_i + \cancel{\nabla_i X_{2d}} + h_0 JX_i -\nabla_{2d} X_i -\cancel{\nabla_i X_{2d}} -a_i X_1\right) - 2\dot{\alpha}_{ij}X_i  \\
& = - 2\dot{\alpha}_{ij}X_i,
\end{align}
where in the last line we used we used our choice of a parallel transported frame. Then $\alpha_{ij}$ is a constant $2d \times (2d-1)$ matrix. Since the first row is zero, we can consider $\alpha$ as a $(2d-1)\times (2d-1)$ constant matrix. Condition (ii) implies that this matrix is orthogonal. Without loss of generality, we may assume that $\alpha_{ij}(0) = \delta_{ij}$. Thus
\begin{equation}
E_{c_j} = \partial_j + a_j \partial_0, \qquad j =2,\ldots,2d.
\end{equation}

\subsubsection*{Step 5: $F_{c_j}$}
For $j=2,\ldots,2d$, the element $F_{c_j}$ is obtained from $E_{c_j}$ by the structural equations:
\begin{align}
F_{c_j} & = -\dot{E}_{c_j} = -\dot{\partial}_j - \dot{a}_j \partial_0 - a_j \dot{\partial}_0  \\
& = X_j + 2g(\tau(X_{2d}),X_j)\partial_0 - h_0 g(JX_i,X_j) \partial_i + g(\nabla_{2d} X_i,X_j) \partial_i + g(\nabla_j X_i,X_{2d}) \partial_i \\
& \quad -\dot{a}_j\partial_0-a_j \partial_1  \\
& = X_j + \left(2g(\tau(X_{2d}),X_j) - \dot{a}_j\right)\partial_0 - h_0 g(JX_i,X_j) \partial_i - g(X_i,\nabla_{2d} X_j) \partial_i \\
& \quad + g(\nabla_j X_i,X_{2d}) \partial_i - a_j\partial_1  \\
& = X_j + \left(2g(\tau(X_{2d}),X_j) - \dot{a}_j\right)\partial_0 - h_0 g(JX_i,X_j) \partial_i - \frac{1}{2}g(X_i,h_0 JX_j - a_j X_1) \partial_i \\
& \quad + g(\nabla_j X_i,X_{2d}) \partial_i - a_j\partial_1  \\
& = X_j + \left(2g(\tau(X_{2d}),X_j) - \dot{a}_j\right)\partial_0 + g(\nabla_j X_i,X_{2d}) \partial_i + \frac{1}{2}h_0 g(X_i,JX_j) \partial_i   -\frac{1}{2} a_j\partial_1,
\end{align}
where we used again the parallel transported frame to eliminate $\nabla_{2d} X_j$ for $j=2,\ldots,2d$.

The computations so far prove Theorem~\ref{t:cansplitting}. In fact the canonical subspaces are defined by
\begin{equation}
S^a  := \spn\{\pi_* F_a\}, \qquad S^b  := \spn\{\pi_* F_b\}, \qquad S^c  := \spn\{\pi_* F_{c_2},\ldots,\pi_* F_{c_{2d}}\},
\end{equation}
(see Sec.~\ref{s:invariantspaces}) and from the computations above we have explicitly:
\begin{align}
X_a & := \pi_* F_a  = -\pi_* \dot{F}_b = X_0 +2Q(\tanf,\tanf)-h_0T = X_0 - a_j X_j, \\
X_b & := \pi_* F_b  = -J\tanf, \\ 
X_{c_j}& := \pi_* F_{c_j}  = X_j, \qquad j=2,\ldots,2d.
\end{align}

\subsubsection*{Step 6: $\mathfrak{R}^{bb}$, $\mathfrak{R}^{bc}$ and $\mathfrak{R}^{cc}$}

To compute the curvature we need to compute symplectic products.
\begin{lemma}\label{l:sympl2}
	Let $X_0,\ldots,X_{2d}$ an adapted frame. Then, along the extremal
\begin{align}
\sigma(X_\mu,X_\nu) & = h_0 \omega(\Tor(X_\mu,X_\nu)) - g(X_{2d},\nabla_\mu X_\nu - \nabla_\nu X_\mu- \Tor(X_\mu,X_\nu)), \\
\sigma(\partial_\mu,X_\nu) & = \delta_{\mu\nu}, \\
\sigma(\partial_\mu,\partial_\nu) & = 0 .
\end{align}
for all $\mu,\nu=0,\ldots,2d$.
\end{lemma}
\begin{proof}
From Eq.~\eqref{eq:sympl}, we obtain, along the extremal
\begin{align}
\sigma(X_\mu,X_\nu) & = h_\alpha d\nu^\alpha(X_\mu,X_\nu) = - h_\alpha \nu^\alpha([X_\mu,X_\nu])  \\
& = -h_\alpha \nu^\alpha(\nabla_\mu X_\nu - \nabla_\nu X_\mu - \Tor(X_\mu,X_\nu))  \\
& = h_0 \omega(\Tor(X_\mu,X_\nu)) - g(X_{2d},\nabla_\mu X_\nu - \nabla_\nu X_\mu - \Tor(X_\mu,X_\nu)).
\end{align}
where, in the last line, we used the fact that the frame is adapted. The second and third identities follow with analogous but shorter computations starting from Eq.~\eqref{eq:sympl}.
\end{proof}

Now we can compute $\mathfrak{R}^{bb}$. It is convenient to split the computation as follows.
\begin{equation}
\mathfrak{R}^{bb} = \sigma(\dot{F}_b,F_b) = \sigma(\dot{F}_b^h,F_b^h) + \sigma(\dot{F}_b^h,F_b^v) + \sigma(\dot{F}_b^v,F_b^h).
\end{equation}
We compute the three pieces, using Lemma~\ref{l:sympl2}. We obtain, after some computations
\begin{align}
\sigma(\dot{F}_b^h,F_b^h) & = - g(\tau(X_{2d}),X_1) + h_0^2 + g(X_{2d},\nabla_{2Q(X_{2d},X_{2d})-h_0 X_{2d} + X_0} X_1) \\
& \quad + g(\nabla_1 X_{2d},2Q(X_{2d},X_{2d})), \\
\sigma(\dot{F}_b^v,F_b^h) & = 2g(\tau(X_{2d}),X_1) +g(\nabla_{2d} \nabla_1 X_1 - \nabla_1 \nabla_{2d} X_1.X_{2d}) + \|Q(X_{2d},X_{2d})\|^2 \\
& \quad - g(\nabla_{Q(X_{2d},X_{2d})-\nabla_1 X_{2d}} X_1,X_{2d}), \\
\sigma(\dot{F}_b^h,F_b^v) & = 2g(\tau(X_{2d}),X_1) + 2\|Q(X_{2d},X_{2d})\|^2 - g(\nabla_1 X_{2d}, 2Q(X_{2d},X_{2d})).
\end{align}
Taking in account that $Q(X_{2d},X_{2d}) - h_0 X_{2d} +\nabla_1 X_{2d} + X_0 = -[X_{2d},X_1]$ (easily proved using the properties of Tanno connection) we obtain immediately
\begin{equation}
\mathfrak{R}^{bb} = R(\tanf,J\tanf,J\tanf,\tanf) + 3\|Q(\tanf,\tanf)\|^2 - 3g(\tau(\tanf),J\tanf)   + h_0^2. 
\end{equation}
where we replaced the explicit tangent vector $\tanf =\dot{\gamma}_\lambda$. We proceed with $\mathfrak{R}^{cc}$. Again we split:
\begin{equation}
\mathfrak{R}^{cc}_{ij} = \sigma(\dot{F}_{c_i}, F_{c_j}) = \sigma(\dot{F}_{c_i}^h,F_{c_j}^h)+\sigma(\dot{F}_{c_i}^v,F_{c_j}^h)+\sigma(\dot{F}_{c_i}^h,F_{c_j}^v) = \sigma(\dot{F}_{c_i}^v,F_{c_j}^h),
\end{equation}
where we used the fact that, by construction, $\dot{F}_{c_i}$ is vertical hence $\dot{F}_{c_i}^h =0$. We use Leibniz's rule, Lemma~\ref{l:prelcomps}, Eq.~\eqref{eq:partrans2} for the parallel transported frame and the identity $[X_{2d},X_i] = \nabla_{2d} X_i - \nabla_i X_{2d}$, valid for $i=2,\ldots,2d$. We obtain, after some computations
\begin{align}
\sigma(\dot{F}_{c_i}^v,F_{c_j}^h) & =  R(X_{2d},X_i,X_j,X_{2d})+ \frac{1}{2}g(\tau(X_{2d}),X_{2d})g(X_j,JX_i) \\
& \quad +  \frac{1}{2}h_0 g(X_j,Q(X_i,X_{2d}))- h_0 g(X_j,Q(X_{2d},X_i)) + \frac{1}{4}h_0^2g(X_i,X_j) - \frac{1}{4}a_ia_j.
\end{align}
for $i,j=2,\ldots,2d$. This expression is symmetric w.r.t. to $i$ and $j$ as a consequence of Lemma~\ref{l:propcurv}. Restoring the tangent vector $\tanf = \dot{\gamma}_\lambda$ and the functions $a_j$ we get, for $i,j=2,\ldots,2d$
\begin{align}
\mathfrak{R}^{cc}_{ij} & =  \mathcal{S}[R(\tanf,X_{i},X_{j},\tanf)]+h_0 \mathcal{S}[g(\tanf,Q(X_{j},X_{i}))] +\frac{1}{4}h_0^2 g(X_i,X_j) \\
& \quad - \frac{1}{4} g(X_i,h_0\tanf -2Q(\tanf,\tanf))g(X_j,h_0\tanf -2Q(\tanf,\tanf)).
\end{align}
The trace over the subspace $S^c= \spn\{X_2,\ldots,X_{2d}\}$ is
\begin{align}
\mathfrak{Ric}^c = \sum_{i=2}^{2d} \mathfrak{R}^{cc}_{ij} & = \mathrm{Ric}(\tanf) - R(\tanf,J\tanf,J\tanf,\tanf) + h_0 g(\tanf,\trace Q) + \frac{1}{4}h_0^2(2d-2) - \|Q(\tanf,\tanf)\|^2 \\
& = \mathrm{Ric}(\tanf) - R(\tanf,J\tanf,J\tanf,\tanf) + \frac{1}{4}h_0^2(2d-2) - \|Q(\tanf,\tanf)\|^2.
\end{align}
where we used property c) and i) of Tanno's tensor. 

Now we compute $\mathfrak{R}^{cb} :S^c \to S^b$ (this is a $(2d-1)\times 1$ matrix). As usual we split
\begin{equation}
\mathfrak{R}^{cb}_j = \sigma(\dot{F}_{c_j},F_b) = \sigma(\dot{F}^v_{c_j},F^h_b).
\end{equation}
We use Leibniz's rule, Lemma~\ref{l:prelcomps}, Eq.~\eqref{eq:partrans2} for the parallel transported frame and the identity $[X_{2d},X_i] = \nabla_{2d} X_i - \nabla_i X_{2d}$, valid for $i=2,\ldots,2d$. Moreover, we use also properties b), d) and h) of Tanno's tensor. After some computations, and restoring the notation $\tanf= \dot\gamma$ we get
\begin{align}
\mathfrak{R}^{cb}_j & = -R(\tanf,X_j,J\tanf,T) + 2g(\tau(\tanf),X_j) - 2g(\tau(\tanf),\tanf)g(\tanf,X_j) + 3g((\nabla_\tanf Q)(\tanf,\tanf),X_j) \\
& \quad  +8h_0g(Q(\tanf,\tanf),JX_j).
\end{align}
By construction, $\mathfrak{R}^{bc} = (\mathfrak{R}^{cb})^*$. We can write, in a more symmetric fashion using Eq.~\eqref{eq:id1}:
\begin{align}
\mathfrak{R}^{bc}_j & = -R(\tanf,J\tanf,X_j,\tanf) +  g(\tau(\tanf),X_j) - g(\tau(\tanf),\tanf)g(\tanf,X_j)  + 3g((\nabla_\tanf Q)(\tanf,\tanf),X_j)\\ & \quad +8h_0g(Q(\tanf,\tanf),JX_j).
\end{align}

\subsubsection*{Step 7: $F_a$}
The structural equations give:
\begin{equation}
F_a  = -\dot{F}_b + \mathfrak{R}^{bb} E_b + \sum_{j=2}^{2d} \mathfrak{R}^{bc}_j E_{c_j}.
\end{equation}
After some computations, we obtain
\begin{equation}
F_a = X_0 - a_i X_i + \phi_i \partial_i + \phi_0 \partial_0,
\end{equation}
where
\begin{align}
a_i X_i & =  h_0 X_{2d} - 2Q(X_{2d},X_{2d}), \\
\phi_i X_i & =  -\nabla_{X_0+2Q(X_{2d},X_{2d})} X_{2d} + \tau(X_{2d}) - g(\tau(X_{2d}),X_{2d}) X_{2d} \\
&\quad  - 4h_0 JQ(X_{2d},X_{2d}) + 2(\nabla_{2d}Q)(X_{2d},X_{2d}) - 3\|Q(X_{2d},X_{2d})\|^2 JX_{2d}, \\
\phi_0 & =  2 g((\nabla_{2d} \tau)(X_{2d}),JX_{2d}) - 4h_0 g(\tau(X_{2d}),X_{2d}) + g((\nabla_1\tau)(X_{2d}),X_{2d})  \\ 
&\quad + 2R(X_{2d},JX_{2d},Q(X_{2d},X_{2d}),X_{2d}) + 2g(\tau(X_{2d}),Q(X_{2d},X_{2d})) \\
& \quad -3\nabla_{2d}\|Q(X_{2d},X_{2d})\|^2. 
\end{align}
The vector $\phi:=\phi_i X_i$ is orthogonal to $X_{2d}$ (or, equivalently, $\phi_{2d} = 0$). Clearly $\phi_0$ is a well defined, smooth function along the geodesic. In particular, it makes sense to take the derivative of $\phi_0$ in the direction of the geodesic or, in terms of our adapted frame, $\nabla_{2d} \phi_0$.

To complete the computation, we assume $Q=0$.

\subsubsection*{Step 8: $\mathfrak{R}^{aa}$ (with $Q=0$)}

Many simplification occur. In particular
\begin{align}
a_i X_i  & =  h_0 X_{2d}, \\
\phi_i X_i  & =  -\nabla_{0} X_{2d} + \tau(X_{2d}) - g(\tau(X_{2d}),X_{2d}) X_{2d} , \\
\phi_0 & =  2 g((\nabla_{2d} \tau)(X_{2d}),JX_{2d}) - 4h_0 g(\tau(X_{2d}),X_{2d}) + g((\nabla_1\tau)(X_{2d}),X_{2d}).
\end{align}
As usual, we split
\begin{equation}
\mathfrak{R}^{aa} = \sigma(\dot{F}_a,F_a) = \sigma(\dot{F}_a^v,F_a^h) = \langle \dot{F}_a^v, F_a^h\rangle = \langle \dot{F}_a^v, X_0 - h_0 X_{2d}\rangle = \dot{F}_a^v(h_0) - h_0 \dot{F}_a^v(h_{2d}),
\end{equation}
where we used the fact that $\dot{F}_a$ is vertical by construction. By Lemma~\ref{l:prelcomps}, we obtain
\begin{equation}
\mathfrak{R}^{aa} = g((\nabla_{h_0\tanf-X_0} \tau)(\tanf),\tanf) + 2h_0^2 g(\tau(\tanf),J\tanf) - 2\|\tau(\tanf)\|^2 + 2g(\tau(\tanf),\tanf)^2 + \nabla_\tanf \phi_0.
\end{equation}
It is not convenient to explicitly compute the derivative $\nabla_{2d} \phi_0$ since no simplifications occur.
\begin{remark}
Since $\pi_* F_a = X_0 - h_0 \tanf$, one would expect a term of the form $R(\tanf,X_0-h_0\tanf,X_0-h_0\tanf,\tanf)$ in the expression for $\mathfrak{R}(X_a,X_a)$. However, using the symmetries of the Riemann tensor (with torsion) and the fact that $\nabla X_0 = 0$, we obtain that $R(\tanf,X_0-h_0\tanf,X_0-h_0\tanf,\tanf) = 0$.
\end{remark}

\subsubsection*{Step 9: $\mathfrak{R}^{ac}$ (with $Q=0$)}

The term $\mathfrak{R}^{ac} : S^a \to S^c$ is a $1\times (2d-1)$ matrix. For $j=2,\ldots,2d$
\begin{equation}
\mathfrak{R}^{ac}_j = \sigma(\dot{F}_a,F_{c_j}) = \sigma(\dot{F}_a^v,F_{c_j}^h) = \langle \dot{F}_a^v, X_j\rangle.
\end{equation}
With long computations using Leibniz's rule, the identity $[X_{2d},X_0] = -\nabla_0 X_{2d} + \tau(X_{2d})$ and Lemma~\ref{l:prelcomps}, we obtain for $j=2,	\ldots,2d-1$
\begin{equation}
\mathfrak{R}^{ac}_j=R(\tanf, X_0, X_j, \tanf)+ g((\nabla_\tanf \tau)(\tanf),X_j)	+ 2h_0g(\tau(J\tanf),X_j),
\end{equation}
and $\mathfrak{R}^{ac}_{2d} = 0$.

\appendix
\section{On the Yang-Mills condition} \label{s:app}
We give here an equivalent characterization of the Yang-Mills condition  for contact manifolds.

\begin{proposition}
The Yang-Mills condition \eqref{eq:YMdef}
is equivalent to
\begin{equation}
\sum_{i=1}^{2d}(\nabla_{X_{i}}\tau)(X_{i})=0.
\end{equation}
for every orthonormal frame $X_{1},\ldots,X_{2d}$ of $\distr$.
\end{proposition}
\begin{proof}
Let $X \in \Gamma(\distr)$. For $Y \in \Gamma(\distr)$, using the properties of Tanno connection
\begin{align}
(\nabla_X T)(X,Y) & = X(T(X,Y))-T(\nabla_X X,Y) - T(X,\nabla_X Y) \\
& =X(d\omega(X,Y)X_0) - d\omega(\nabla_X X,Y)X_0 - d\omega(X,\nabla_X Y)X_0 \\
& = [X(g(X,JY)) - g(\nabla_X X,JY) - g(X,J\nabla_X Y)]X_0 \\
& = [-X(g(JX,Y)) + g(J\nabla_X X,Y) + g(JX,\nabla_X Y)]X_0 \\
& = - g((\nabla_X J)X,Y)X_0 = - g(Q(X,X),Y)X_0.
\end{align}
On the other hand, for $Y= X_0$ we get
\begin{align}
(\nabla_X T)(X,X_0) & = X(T(X,X_0))-T(\nabla_X X,X_0) - T(X,\nabla_X X_0) \\
& = - X(\tau(X)) +\tau(\nabla_X X) = -(\nabla_X \tau)(X).
\end{align}
Then, for any $X \in \Gamma(\distr)$ and $Y \in \Gamma(TM)$, we have the following identity
\begin{equation}
(\nabla_X T)(X,Y) = -g(Q(X,X),Y)X_0 - \omega(Y)(\nabla_X \tau)(X).
\end{equation}
Taking the trace, and using that $\trace Q=0$ by item i) of Lemma \ref{lem:identities}, we get the result.
\end{proof}

\bibliographystyle{abbrv}
\bibliography{biblio/Contact-Biblio}

\end{document}